\documentclass[a4paper]{article}

\usepackage[shortlabels]{enumitem}
\usepackage[]{geometry}
\usepackage[utf8]{inputenc}
\usepackage{amssymb,amsmath,amsthm}
\usepackage{hyperref}
\usepackage[usenames, dvipsnames]{color}
\usepackage[normalem]{ulem}
\usepackage{tikz}
\date{\today}

\newcommand{\diff}[0]{\,{\rm d}}
\newcommand{\arctanf}[0]{\arctan_{\chi}}

\newcommand{\tanf}[0]{\tan_{\chi}}
\newcommand{\sinf}[0]{\sin_{\chi}}
\newcommand{\cosf}[0]{\cos_{\chi}}
\newcommand{\tanfab}[0]{\tan_{\chi_{a,b}}}
\newcommand{\sinfab}[0]{\sin_{\chi_{a,b}}}
\newcommand{\cosfab}[0]{\cos_{\chi_{a,b}}}

\newcommand{\pif}[0]{\pi_{\chi}}
\newcommand{\piftwo}[0]{\pi_{\chi}^\perp}

\newcommand{\piffour}[0]{\pi_{\chi}^\angle}
\newcommand{\Betaf}[0]{{\rm B}_{\chi}}

\definecolor{rougeG}{rgb}{.76,0,.12}
\definecolor{vertG}{rgb}{.07,.56,.25}

\numberwithin{equation}{section}
\usepackage{authblk}

\title{
A new family of deformed trigonometric functions and their corresponding Pythagorean-like relations: application to the nonlinear harmonic oscillator\\
}
\author[1]{P. R. Gordoa}

\author[1]{A. Pickering}

\author[* 1,2]{D. Puertas-Centeno}

\author[1]{E. V. Toranzo}

\affil[1]{Departamento de Matemática Aplicada, Ciencia e Ingeniería de los Materiales y Tecnología Electrónica, Universidad Rey Juan Carlos,
	28933 Móstoles (Madrid), Spain}
\medskip
\affil[2]{Data, Complex Networks and Cybersecurity Research Institute, Universidad Rey Juan Carlos, 28028 (Madrid), Spain}
\affil[*] {Corresponding author: david.puertas@urjc.es}

\usepackage{blindtext}
\newtheorem{definition}{Definition}

\newtheorem{coro}{Corollary}
\newtheorem{remark}{Remark}
\newtheorem{teorema}{Theorem}
\newtheorem{prop}{Proposition}

\usepackage[T1]{fontenc}
\usepackage{palatino}
\begin{document}
	\maketitle
\begin{abstract}
	In this work we present an extension of the classical trigonometric functions which include, amongst others, the Lindqvist $p$-trigonometric functions. Additionally, we show their relevance in the study of certain versions of the nonlinear harmonic oscillator equation. Furthermore, we study their more important  mathematical properties and obtain their corresponding Pythagorean-like relations. We also define a new extension of the Beta and Ateb sine functions using these new $\chi$-trigonometric functions. Indeed, these new Beta and Ateb functions are involved in the closed form  expressions obtained for the solutions of the nonlinear oscillator under certain constraints. 
\end{abstract}
\section{Introduction}
The study of different extensions of the trigonometric functions to a more general scenario was carried out by different authors some decades ago~\cite{Shelupsky1959,Burgoyne1964}. However, the original (as far as we know) definition of Shelupsky in 1959~\cite{Shelupsky1959} only began to garner research interest  with  the study of Lindqvist in 1995~\cite{Lindqvist1995}, where this extended family of trigonometric function was used to solve an eigenvalue problem involving the $p$-Laplacian operator. We recall the definition of the monoparametric Lindqvist $p$-arctan function \cite{Lindqvist1995,Edmunds2009}
\begin{equation}\label{eq:p-arcs}
	\arctan_p(x)=\int_0^x\frac{\diff t}{1+t^p}, \qquad x\in [0,\infty), \quad 1<p<\infty\, .
\end{equation}
 Given this  definiton, the $p$-tangent function can be defined as the inverse function of the $p$-arctan function. Further extensions in this framework have been proposed by Drábek and Manásevich~\cite{Drabek1999} through a biparametric version of the Lindqvist $p$-arcsine function, defined as

\begin{equation}\label{eq:arcsinpq}
	\arcsin_{p,q}(x)=\int_0^x\frac{\diff t}{(1-t^q)^{1/p}}=\frac1q {\rm B}\left(\frac1q,\,1-\frac{1}p;\,x^q\right), \qquad x\in [0,1],
\end{equation}
where B$(a,b;z)$ denotes the incomplete Beta function,
\begin{equation}\label{eq:Beta0}
	B(a,b;z)=\int_0^z t^{a-1}(1-t)^{b-1}\diff t, \quad a,b>0,\,\,\, z\in[0,1]\, ,
\end{equation}
which for $z=1$ reduces to the (complete) Beta function, i.e., $B(a,b;1) \equiv B(a,b) = \int_0^1 t^{a-1}(1-t)^{b-1}\diff t$.
 The $\sin_{p,q}$ function is defined as the inverse of the $\arcsin_{p,q}$ function, and the $\cos_{p,q}$ function is defined as the derivative of $\sin_{p,q}$. Then the following  Pythagorean-like identity is obtained from the definition in Eq.~\eqref{eq:arcsinpq},
\begin{equation}\label{eq:q-pyth}
	\sin^q_{p,q}(x)+\cos^p_{p,q}(x)=1,\quad x\in \left[0,\frac{\pi_{p,q}}{2}  \right] 
	\end{equation}
	where $\pi_{p,q} = \frac{2}{q}B\left(\frac{1}{q},\frac{1}{p^*} \right)$ with $p^* = \frac{p}{p-1}$. Moreover, when $p=q,$
 with the notation $\sin_{p,p}=\sin_p$,  it  follows that 
	$$\tan_p(x)=\frac{\sin_p(x)}{\cos_p(x)},$$
	and then the identity~\eqref{eq:q-pyth} can be rewritten as
	\begin{equation}
		1+\tan^p_p(x)=\frac1{\cos^p_p(x)}.
	\end{equation}
 These  functions not only have mathematical interest~\cite{Miyakawa2021,Miyakawa2022,Takeuchi2021}, but also practical  interest,  as they have been proved to be useful tools in the description of nonlinear classical~\cite{Cveticanin2020a, Skrzypacz2020,Vujkov2022,Wei2012} and quantum~\cite{Shababi2016,Ahrami2024,Gordoa2025} systems. In this work we focus,  indeed, on  their application to the nonlinear harmonic oscillator~\cite{Cveticanin2020b, Ghosh2019,Cveticanin2025a,Cveticanin2025b}, described by the following diferential equation \cite{Cveticanin2020b},
 	\begin{equation}\label{eq:nonlinear_osc0}
 		\ddot x+\omega_\alpha^2 x^{\alpha}=0,
 	\end{equation}
 	and further extensions of  this last of  the form
 	\begin{equation}\label{eq:nonlinear_osc_f}
 		\ddot x+\omega_\alpha^2 x^{\alpha}=f(x,\dot x),
 	\end{equation}
 	under certain restrictions of the function $f$.

An alternative useful approach is through the use of the so-called Ateb trigonometric funcions, originally defined in~\cite{Rosenberg1963}. From~\cite{Senik1969,Rebot2023} and after some algebraic manipulations one can write the Ateb sine function as
\begin{equation}
	{\rm sa}(n,m;x)=\left[{\rm B}^{-1}\left(\frac1{n+1},\,\frac{1}{m+1};2x\right)\right]^\frac 1{n+1}.
\end{equation}
For convenience we will adopt the notation ${\rm sa}(n;x)={\rm sa}(n,1;x).$


  
  In a totally different context, the extension of the exponential and logarithmic functions is central in the description of non-extensive thermodinamical systems\cite{Tsallis1988,Tsallis2023}.
Indeed, the use of  a  generalized logarithm function with suitable properties to construct generalized non-additive entropies is an interesting and fruitful research area~\cite{Tempesta2011,Tempesta2019}.  Within this framework, some authors have proposed the following definition of a deformed logarithm function.  Given a function $\chi:\mathbb{R}^+ \longrightarrow \mathbb{R}^+$ increasing on $\mathbb{R}^+$ (known as  the  deformation function)  the deformed $\chi$-logarithm is defined as~\cite{Naudts2004,Zhang2015,Matsuzoe2015,Matsuzoe2016}
  	\begin{equation}\label{eq:chi_log}
  		\log_\chi(x)=\int_1^x \frac{\diff t}{\chi(t)}, \quad x\in \mathbb{R}^+.
  	\end{equation}
  	Note that it is an increasing and concave function on $\mathbb{R}^+$
  	(we could consider the case of $\chi$ non-increasing, if we were to admit a non-concave logarithm).
  	 
  	 In particular, the Tsallis $q$-logarithm is  recovered  when $\chi(t)=t^q$, 
  		\begin{equation}\label{eq:chi_log_q}
  		\log_q(x)=\int_1^x \frac{\diff t}{t^q} =\frac{x^{1-q}-1}{1-q},
  	\end{equation}
 whose inverse is known as the Tsallis $q$-exponential function given by
\begin{equation}
	\exp_q(x)=\left(1+(1-q)x\right)_+^\frac1{1-q},
\end{equation}  	
where $(\cdot)_+ = \max\{0,\cdot\}$, which mainly corresponds with the so-called $q$-stretched Gaussian, or generalized Gaussian, ubiquitous in the study of nonextensive systems \cite{Tsallis2023}. It is worth  mentioning  that these generalized $q$-stretched Gaussians are also the extremizers of certain generalized families of informational inequalities  of  Cramér-Rao, Stam and moment-entropy types~\cite{Lutwak2005,Bercher2012}.  In this regard, it has been recently found that the Lindqvist $p$-trigonometric functions are closely related with the extremizers of further extensions of the mentioned informational inequalities~\cite{Puertas2025,IP2026b}



In this work we propose a further extension of the trigonometric functions by using the $\chi$-logarithm framework in Eq~\eqref{eq:chi_log} to define a deformed $\chi$-arctan function and the subsequent $\chi$-trigonometric functions. In this scenario the Lindqvist $p$-trigonometric functions appear as the natural trigonometric counterpart of the Tsallis $q$-exponential and $q$-logarithmic functions. In addition, their corresonding $\chi$-Pythagorean like relations are written through the use of both classical and generalized logarithmic functions. Moreover, we show the utility of this new family of trigonometric function in the study of the nonlinear oscillator. To do so, we define a new extension of the incomplete Beta function, as well  of  the Ateb sine function,  which   allow us to express, in closed form, the exact solution of certain family of nonlinear differential equations. To accomplish this, the generalized Sundman-like transformations of power type, recently introduced  by the authors~\cite{Gordoa2026}, will play a crucial role. \\

The structure of the work is the following: in Section 2 we present the main definitions of the $\chi$-trigonometric functions and we study  their  main mathematical properties. In Section 3 we show the counterpart of the Pythagorean-like relations in the $\chi$-deformed framework. We use them to define the corresponding  $\chi$-arcsine  and $\chi$-arccosine functions. In Section 4 we detail some remarkable examples. In Section 5 we define the $\chi$-deformed counterpart of the incomplete Beta and Ateb functions and we show their relation with certain families of autonomous second order differential equations. In addition, we review the definition of the generalized Sundman transformations. In Section 6, we show the utility of these new $\chi$-Beta and $\chi$-Ateb functions  in expressing  the exact solution of the nonlinear harmonic oscillator under certain restrictions  on the potential function.

\section{$\chi$-trigonometric functions}
In this section we introduce the main notions of the paper. More precisely, we first define the generalized $\chi$-arctan and $\chi$-tan functions, which in turn will let us to deal with the  definitions of the $\chi$-sine and $\chi$-cosine functions. We study some basic properties such as integral expressions, the derivative of the $\chi$-cosine function and their implication in simple differential equations. In addition, we also give the definition of a pair of perpendicular-like and oblique-like ''$\pi_\chi$'' constants.
\subsection{The $\chi$-arctan and $\chi$-tan functions}
\begin{definition}\label{def:chi_arctan}
	Let $\chi$ be a real positive and increasing function defined on $\mathbb{R}^+\cup\{0\}$. We define the $\arctanf$ function as 
	\begin{equation}\label{eq:chi_arctan}
		\arctanf(x)=\int_0^x\frac{\diff t}{1+\chi^2(t)}, \quad x \in \mathbb{R}^{+}\cup\{0\}\, .
			\end{equation}
	As $\arctanf$ is an increasing function, it has a well-defined inverse which will be denoted as $\tanf$.
\end{definition}

It follows from the above that 
\begin{equation}\label{eq:tanf_der}
	\tanf' (x)=1+\chi^2\left(\tanf(x)\right)\, ,
\end{equation}
which means that $\tanf$ is an increasing function on its domain as well, as  we might wish. Trivially, from the definition  it directly follows that  $\tanf(0)=0.$
In addition,  one can define the following pair of "$\pi_\chi$" constants, which will be very useful throught the paper
\begin{equation}
	\pif^\perp=2\,\arctanf(\infty)\leqslant \infty
\end{equation}
and 
\begin{equation}
	\pif^\angle=4\,\arctanf(1)\leqslant 4,
\end{equation}
for any function $\chi$ which, in general, are  different. We refer to them as perpendicular and oblique "$\pi_\chi$" constants, respectively. In particular, after simple algebraic manipulations the following relation is  obtained 
\begin{equation}
	\piftwo-\piffour=2\left(\int_1^\infty\frac{\diff t}{1+\chi^2(t)}-\int_0^1\frac{\diff t}{1+\chi^2(t)}\right).
\end{equation}
Note that, when $\chi(t)=t$, both perpendicular and oblique constants are equal and the classical $\pi$ constant is recovered. 

From the definitions of the $\pi_\chi$ constants as well as the $\arctanf$, we infer that, in the case  that $\pif^\perp < \infty$, 
	\[
	\arctanf(x) : \mathbb{R} \mapsto \Omega_{\chi} = \left[0,\frac{\pif^\perp}{2} \right) \subset \mathbb{R}\, ,
	\]
	i.e., $\tanf$ has bounded support. In this case, we can always extend this function to $\left(-\frac{\pif^\perp}{2} ,\frac{\pif^\perp}{2} \right)$ by oddness ($\tanf(x) = -\tanf(-x)$) and to $\mathbb{R}$ by $\pif^\perp$-periodicity, which imposes that, since $\tanf(0) = 0$,
	\[
	\tanf\left(n\pif^\perp \right) = 0 , \quad \forall n\in\mathbb{Z}\, .
	\]
	It is worth  highlighting  that such  a  periodic extension comes naturally as long as $\chi$ has definite parity.

\begin{remark}The classical and generalized Lindqvist tangent functions~\cite{Lindqvist1995} are recovered when $\chi(t)=t$ and $\chi(t)=t^{p/2}$, respectively. 
More precisely, adopting the notation $\chi_p(t):=t^{p/2}$, it follows that

 \[
 \arctan_{\chi_{p}}(x)=\int_0^x\frac{\diff s}{1+s^{p}} = \arctan_{p}(x) \,, \quad x\in [0,\infty) \,.
 \]
 
 Then,
 \begin{equation}\label{eq:tan_p}
 	\tan_{\chi_{p}}(x)= \tan_p(x)\, , \quad x \in \Omega_{\chi_{p}} = \left[0, \frac{\pi^{\perp}_{\chi_{p}}}{2}\right)\, ,
 \end{equation}
 where
 \begin{equation}\label{eq:pi_p}
 \pi^{\perp}_{\chi_{p}}\equiv \pi_p =  2\int_0^{\infty}\frac{\diff s}{1+s^{p}} = \frac{2}{p}B\left(\frac{1}{p},\frac{1}{p^*} \right), \quad p^* = \frac{p}{p-1}\, .
 \end{equation}
\end{remark}

In addition, $\pi^{\angle}_{\chi_{p}}=4\arctan_p(1).$ Let us recall that, $\tan_p(\piffour/4)=1$ is equivalent to $\cos_p(\piffour/4)=\sin_p(\piffour/4).$ As a consequence of~\eqref{eq:q-pyth} with $q=p$, the latter implies $\sin_p(\piffour/4)=1/2^p.$ Then $\pi^{\angle}_{\chi_{p}}=4\arcsin_p(2^{-1/p})$
$$\pi^{\angle}_{\chi_{p}}=\frac 4 p {\rm B}\left(\frac 1p,\,\frac 1{p^*},\frac 12\right).$$
Note that, for any $p\neq 2$, one has that $\pi_{\chi_{p}}^\angle\neq \pi_p.$  

\subsection{The $\chi$-sine and $\chi$-cosine functions}
We define the real functions $\sinf$ and $\cosf$ via the relations
\begin{equation}\label{eq:scsat}
	\frac{\sinf(x)}{\sinf'(x)}=\tanf(x),\quad \cosf(x)=\sinf'(x).
\end{equation}
It immediately follows that  
\begin{equation}\label{dlogsinf}
	\left(\log(\sinf(x))\right)'=\frac1{\tanf(x)},
\end{equation}
which inspires the following definition.

\begin{definition}[$\chi$-sine and $\chi$-cosine functions]
	The $\chi$-sine and $\chi$-cosine functions are defined on $\overline{\Omega}_{\chi} = \left[0, \piftwo/2 \right]$ as
	\begin{equation}\label{eq:sine_def}
	\sinf(x):=\exp\left(\int_{\piftwo/2}^x \frac{\diff t}{\tanf(t)}\right)
\end{equation}
and
	\begin{equation}\label{eq:cosf_def1}
		\cosf(x):=\sinf'(x)=\frac{\sinf(x)}{\tanf(x)}, 
	\end{equation}
	respectively.
\end{definition}

\begin{remark}\label{remark:pi2}
	The choice of $\pif/2$ in the limit of the integral enssures that 
	\begin{equation}
	\sinf(\piftwo/2)=1.
	\end{equation}

	In addition, $\tanf(\piftwo/2)=\infty$ by construction. Then, taking into account the definition~\eqref{eq:cosf_def1}  it follows that  
\begin{equation}
	\cosf(\piftwo/2)=0.
\end{equation}

\begin{remark}
	By definition, both $\sinf$ and $\cosf$ are nonnegative functions on $\overline{\Omega}_{\chi}$, and $\sinf$ is increasing on $\Omega_{\chi}$.
\end{remark}

		\begin{remark}\label{remark:q-case1} 
		In a similar way  to  the $\tanf$ funtion, in the case that  $\pif^\perp<\infty$, we can extend the domain of both  the  $\sinf$ and $\cosf$ functions to $\mathbb{R}$ by the following rules:
			\begin{itemize}
				\item Extension to the domain $x\in [0,	\pif^\perp]$
				$$\forall x\in \left( \frac{\pif^\perp}{2}, \pif^\perp\right],\qquad \sinf(x) := \sinf(\pif^\perp - x),\quad{\rm and} \quad \cosf(x) := -\cosf(\pif^\perp - x).$$
				\item Extension to the domain $x\in [-\pif^\perp,	\pif^\perp]$
				$$\forall x\in [-\pif^\perp,0),\quad\sinf(-x): = -\sinf(x),\quad{\rm and} \quad \cosf(-x) := \cosf(x).$$
				\item Extension to the plain domain $x\in\mathbb R$
				$$\forall n\in\mathbb Z,\;\forall x\in\mathbb R,\quad \sinf(x) := \sinf(x+2n\pif^\perp),\quad{\rm and} \quad \cosf(x) := \cosf(x+2n\pif^\perp).$$
			\end{itemize}
	\end{remark}  
\end{remark}

	\begin{remark}
	By definition, one has that $\tanf(\piffour/4)=1$, that is, 
	\begin{equation}
		\cosf\left(\frac{\piffour}4\right)=\sinf\left(\frac{\piffour}4\right)=\exp\left(\int_{\piftwo/2}^{\piffour/4} \frac{\diff t}{\tanf(t)}\right).
	\end{equation}
In addition, the change of variable $s=\tanf(t)$ gives
		\begin{equation}\label{eq:cospi4}
		\cosf\left(\frac{\piffour}4\right)=\sinf\left(\frac{\piffour}4\right)=\exp\left(-\int_1^{\infty} \frac{\diff s}{s\,(1+\chi^2(s))}\right):=\mathcal K_\chi,
	\end{equation}
	where we have introduced the notation $\mathcal K_\chi$ to make  reading easier .
	From the latter equation, it is natural to impose the condition 
\begin{equation}\label{eq:cond_pi4}
	\int_1^\infty \frac{\diff s}{s(1+\chi^2(s))}<\infty
\end{equation}	
	to ensure that the constant $\mathcal K_\chi$ be different  from   zero.
		A simple convergence study shows that if $\chi$ is an increasing algebraic function the latter integral is always convergent. In addition, if $\chi(t)\sim\log^p(t)$ when $t>>1$, then the integral is convergent for $p>1/2$.  
\end{remark}
%

\begin{remark}
		
	Letting $s=\tanf(t)$ in Eq.~\eqref{eq:sine_def}, and recalling that $\tanf(\piftwo/2)=\infty$, one can express the $\chi$-sine function as 
		\begin{equation}
			\sinf(x)=\exp\left(\int_\infty^{\tanf (x)} \frac{\diff s}{s(1+\chi^2(s))}\right)
		\end{equation}
	When we evaluate the $\chi$-sine function  at zero it follows that 
	
	\begin{equation}
		\sinf(0)=\exp\left(-\int_{0}^\infty \frac{\diff s}{s(1+\chi^2(s))}\right).
	\end{equation}
As a consequence, if
	$$\int_{0}^\infty \frac{\diff s}{s(1+\chi^2(s))}<\infty
	$$
	then  $\sinf(0)>0,$ and consequently $\cosf(0)=\frac{\sinf(0)}{\tanf(0)}=\infty,$ which prevents  a correct definition of the $\chi$-cosine function at  zero.   However,  if we assume that  
	 \begin{equation}\label{cond:0}
	 \int_{0}^\infty \frac{\diff s}{s(1+\chi^2(s))}=\infty
	 \end{equation}
	  
	 we obtain $\sinf(0)=0$, and then $\cosf(0)$ is an indetermination which can be solved as we show below.
	\end{remark}

	 In general, we will assume the following conditions throughout the paper:
	\begin{equation}\label{eq:conditions}
		\int_{0}^1 \frac{\diff s}{s(1+\chi^2(s))}=\infty,\quad 	 	\int_1^\infty \frac{\diff s}{s(1+\chi^2(s))}<\infty,\quad 	 	\int_{0}^1 \frac{\chi^2(s)\diff s}{s(1+\chi^2(s))}<\infty.
	\end{equation}
	The first condition ensures that $\sinf(0)=0,$ the second condition ensures that $\sinf(\piffour/4)\neq0$ and the third condition ensures that  $\cosf(0)$ be finite.  In addition, the first and the last conditions jointly impose that $\chi(0)=0$.
\begin{prop}
	Assuming in force the conditions in Eq.~\eqref{eq:conditions}  it  follows that 
		
	\begin{equation}\label{eq:sin0cos0}
		\sinf(0)=0,\qquad \cosf(0)=\exp\left(\int_0^1 \frac{\chi^2(s)\diff s}{s\left(1+\chi^2(s)\right)}-\int_1^\infty \frac{\diff s}{s\left(1+\chi^2(s)\right)}\right)
	\end{equation}
	Moreover, the following expression for the cosine function is fulfilled

	\begin{align}\label{eq:cosin_alt_exp}
		\cosf(x)&=
		\exp\left(-\int_1^{\infty} \frac{\diff s}{s(1+\chi^2(s))}-\int_{1}^{\tanf (x)} \frac{\chi^2(s)\diff s}{s(1+\chi^2(s))}\right)\nonumber\\
		& = \mathcal K_\chi \exp\left(-\int_{1}^{\tanf (x)} \frac{\chi^2(s)\diff s}{s(1+\chi^2(s))}\right),\quad x\in \overline{\Omega}_{\chi} .
	\end{align}	
	
\end{prop}	  
\begin{proof}
From the condition~\eqref{cond:0}  it directly follows that  $\sinf(0)=0$. We continue with 
	$$
	\cosf(x)=\frac{\sinf(x)}{\tanf(x)}.
	$$
	In addition, recalling that $\tanf(\piffour/4)=1,$ the following identity is fulfilled
	$$
	\int_{\piffour/4}^x \frac{\tanf'(t)\diff t}{\tanf(t)}=\log\left(\tanf(x)\right)-\log\left(\tanf(\piffour/4)\right)=\log\left(\tanf(x)\right).
	$$
	Then, from the latter equality and by using the definitions of the $\chi$-sine and $\chi$-cosine function in Eqs~\eqref{eq:sine_def} and~\eqref{eq:cosf_def1} we can compute
	\begin{eqnarray*}
	\cosf(x)&=&\frac{\sinf(x)}{	\exp\left(\int_{\piffour/4}^x \frac{\tanf'(t)\diff t}{\tanf(t)}\right)}
	\\
	&=&
	\exp\left(\int_{\piftwo/2}^x \frac{\diff t}{\tanf(t)}-\int_{\piffour/4}^x \frac{\tanf'(t)\diff t}{\tanf(t)}\right)
	\\
	&=&
	\exp\left(\int_{\piftwo/2}^{\piffour/4} \frac{\diff t}{\tanf(t)}+\int_{\piffour/4}^x \frac{\diff t}{\tanf(t)}-\int_{\piffour/4}^x \frac{\tanf'(t)\diff t}{\tanf(t)}\right).
	\\
&=&
\exp\left(\int_{\piftwo/2}^{\piffour/4} \frac{\diff t}{\tanf(t)}+\int_{\piffour/4}^x \frac{(1-\tanf'(t))\diff t}{\tanf(t)}\right).
\\
		\end{eqnarray*}
Now, using Eq.~\eqref{eq:tanf_der}  it follows that 		
	\begin{eqnarray*}
	\cosf(x)&=&
	\exp\left(\int_{\piftwo/2}^{\piffour/4} \frac{\diff t}{\tanf(t)}-\int_{\piffour/4}^x \frac{\chi^2(\tanf(t))\diff t}{\tanf(t)}\right).
\end{eqnarray*}	
We perform the change of variable $s=\tanf(t)$ and then $\diff s=(1+\chi^2(\tanf(t))\diff t$. Then, recalling that $\tanf(\piffour/4)=1$ and $\tanf(\piftwo/2)=\infty,$  it follows that 

		\begin{eqnarray*}
		\cosf(x)&=&
		\exp\left(-\int_{1}^{\infty}\frac{\diff s}{s(1+\chi^2(s))}-\int_{1}^{\tanf (x)} \frac{\chi^2(s)\diff s}{s(1+\chi^2(s))}\right),
	\end{eqnarray*}	
which leads  to  Eq.~\eqref{eq:cosin_alt_exp} and the right hand of Eq.~\eqref{eq:sin0cos0}, concluding the proof.
\end{proof}

\subsection{The derivative of the $\chi$-cosine function}
Let us now compute the derivative of the $\cosf$ function. 

\begin{prop}\label{prop:cosf_der} The derivative of the $\chi$-cosine function is given by
	\begin{equation}\label{eq:cosf_der}
		\cosf' (x)=-\frac{\cosf^2(x)}{\sinf(x)}\chi^2\left(\tanf(x)\right), \quad x\in \left(0,\frac{\piftwo}{2}\right)\, ,
	\end{equation}
	which means that $\cosf$ is a decreasing function on $\Omega_{\chi}$ and also that $\sin_{\chi}$ is concave on $\Omega_{\chi}$ .
\end{prop}
\begin{proof}
	
We start by noting that $\sinf(x)=\tanf (x)\cosf(x)$, and then, using  Eq.~\eqref{eq:tanf_der} and Eq.~\eqref{eq:scsat}  it follows that
$$
\cosf(x)=\tanf' (x)\cosf(x)+\tanf(x)\cosf' (x)=\left(1+\chi^2\left(\tanf(x)\right)\right)\cosf(x)+\tanf(x)\cosf' (x),
$$
that is 
$$
0=\chi^2\left(\tanf(x)\right)\cosf(x)+\tanf(x)\cosf' (x)
$$
from which Eq.\eqref{eq:cosf_der} easily follows.

\end{proof}

	\begin{remark} [Two corollaries of Proposition~\ref{prop:cosf_der}]
		Let us start by adopting the notation $y(x)=\sinf(x)$. Then, on one hand Eq.~\eqref{eq:cosf_der} can be written as 
		\begin{equation}
			\frac{y''y}{y'^2} = -\chi^2\left(\frac {y}{y'}\right),
		\end{equation}
		which is an autonomous equation characterizing the $\sinf$ function.
		
		Meanwhile, on the other hand, Eq.~\eqref{eq:cosf_der} can also be written as
		
		\begin{equation}\label{eq:EDO_na}
			\frac{y''y}{y'^2} = -\chi^2\left(\tanf(x)\right).
		\end{equation}
		
		Taking into account that the latter equation is valid for an arbitrary increasing function $\chi$,  it can be chosen in such a way that
		\begin{equation}\label{eq:fg}
			-\chi^2\left(\tanf(x)\right)=g(x),
		\end{equation}
		for an arbitrary and decreasing function $g$. Then, Eq.~\eqref{eq:EDO_na} can be rewritten as 
		\begin{equation}\label{eq:EDO_g}
			\frac{y''y}{y'^2} = g(x),
		\end{equation}
		whose solution is given by $y(x)=\sinf(x)$, subject to~\eqref{eq:fg}.  Using the derivative of the $\tanf$ function as given in Eq.~\eqref{eq:tanf_der}, from~\eqref{eq:fg} it follows that
		\begin{equation}
			\tanf'(x)=1-g(x),
		\end{equation} 
		that is
		\begin{equation}
			\tanf(x)=x-\int g(x)\diff x.
		\end{equation}
		To continue, we note that 
		\begin{equation}
			\frac1{\tanf(x)}=\left(\log(\sinf(x))\right)'  =\frac1{x-\int g(x)\diff x}.
		\end{equation} 
		Finally, we obtain that 
		\begin{equation}
			\sinf(x)=\exp\left(\int\frac{\diff x}{x-\int g(x)\diff x}\right).
		\end{equation}
		Since the solution of Eq.~\eqref{eq:EDO_g} is precisely given by $y(x)=\sinf(x)$, we have then proved that the solution to Eq.~\eqref{eq:EDO_g} is given by
		
		\begin{equation}
			y(x)=\exp\left(\int\frac{\diff x}{x-\int g(x)\diff x}\right).
		\end{equation}
		
	\end{remark}

To conclude this section we highlight the relation with the Lindqvist-type $p$-sine and  $p$-cosine functions

\begin{remark}\label{remark:q-case2} 
	
		Now, we highlight again the case $\chi_{p}=t^{p/2}$. We proceed with the computation of $\sin_{\chi_{p}}$ using Eq.~\eqref{eq:sine_def}:
		\begin{align*}
				\sin_{\chi_{p}}(x) &= \exp\left(\int_{\piftwo/2}^x \frac{\diff t}{\tan_{\chi_{p}}(t)}\right) =  \exp\left(\int_{\pi_p/2}^x \frac{\diff t}{\tan_{p}(t)}\right) = \exp\left(\int_{\pi_p/2}^x \frac{\cos_p (t)}{\sin_{p}(t)}\, dt\right) \\[0.5em]
				& = \exp\left(\int_{\pi_p/2}^x \frac{(\sin_p (t))'}{\sin_{p}(t)}\, dt\right) = \exp\left(\left[\log \sin_p(t)\right]_{\pi_p/2}^{x} \right) = \exp\left(\log \sin_p(x) \right) \\[0.5em]
				& = \sin_p(x)\, , \quad x\in \left[0, \frac{\pi_p}{2}\right].
			\end{align*}
		Then, $\cos_{\chi_{p}}$ is given by
		
			$$
				\cos_{\chi_{p}}(x) = \frac{\sin_{\chi_{p}}(x)}{\tan_{\chi_{p}}(x)}=\frac{\sin_{p}(x)}{\tan_{p}(x)}=\cos_p(x).
			$$
		
\end{remark}

\section{The $\chi$-pythagorean relations and the $\chi$-arcsine and $\chi$-arccos functions}
In this section we show the correspoding $\chi$-Pythagorean-like relations for the $\chi$-trigonometric functions. To do so, we use a complementary deformation function $\overline \chi$ which allow us to express the relation between $\chi$-sine and $\chi$-cosine functions in a very simple way. We also give an alternative expression of the $\chi$-Pythagorean-like relations through a second complementary deformation function $\widetilde \chi.$  We will use them  later to obtain the expresions for the corresponding $\chi$-arcsine and $\chi$-arccosine functions. 
\subsection{The $\chi$-pythagorean relations}
We begin by giving the first version of the $\chi$-Pythagorean relation.
\begin{teorema}
Given $\chi$ satisfying the condition \eqref{eq:cond_pi4}, then  
\begin{equation}\label{eq:gen_pyth_rel1}
	\log\left(\frac{\sinf(x)}{\cosf\left(\frac\piffour4\right)}\right)=\log_{\overline{\chi}}\left(\frac{\sinf(x)}{\cosf(x)}\right), \quad x\in \overline{\Omega}_{\chi}\, ,
\end{equation}
where 
 $\overline \chi(s)=s(1+\chi^2(s))$, which is a positive increasing function for $0\leqslant s \leqslant x$ assuming that $\chi$ is increasing in $(0,\infty)$.
\end{teorema}
\begin{proof}
	We start by noting that, assuming condition~\eqref{eq:cond_pi4}, the generalized $\chi$-sine function can be expressed as
	\begin{eqnarray}
		\sinf(x)&=&\exp\left(\int_{\piftwo/2}^x \frac{\diff t}{\tanf(t)}\right)=\exp\left(\int_{\piftwo/2}^{\piffour/4} \frac{\diff t}{\tanf(t)}\right)\exp\left(\int_{\piffour/4}^x \frac{\diff t}{\tanf(t)}\right)
		\\
		&=&\sinf\left(\frac\piffour4\right)\exp\left(\int_{\piffour/4}^x \frac{\diff t}{\tanf(t)}\right), \quad x\in \overline{\Omega}_{\chi}.
	\end{eqnarray}
	
	The change of variable $s=\tanf(t)$ gives
	\begin{equation}
		\sinf(x)=\sinf\left(\frac\piffour4\right)\exp\left(\int_1^{\tanf(x)} \frac{\diff s}{s(1+\chi(s)^2)}\right),
	\end{equation}
	which can be written as
	\begin{equation}
		\sinf(x)=\sinf\left(\frac\piffour4\right)\exp\left(\log_{\overline{\chi}}\left(\tanf(x)\right)\right),
	\end{equation}
	where $\log_\chi$ is defined in Eq.~\eqref{eq:chi_log}. Finally, condition~\eqref{eq:cond_pi4} ensures that $\sinf\left(\frac\piffour4\right)>0$, and Eq.~\eqref{eq:gen_pyth_rel1} follows after straightforward manipulations, and after interchanging $\sinf(\piffour/4)$ and $\cosf(\piffour/4)$ (as stated in Eq.~\eqref{eq:cospi4}), concluding the proof.
\end{proof}

\begin{remark}	
		The latter equation can be seen as a generalized pythagorean relation for the generalized $\chi$-trigonometric functions. Indeed, letting $\chi(s)=s$, for which choice which these functions reduce to the standard trigonometric ones, and for $x\in \left[0,\frac{\pi}{2}\right]$, Eq.~\eqref{eq:gen_pyth_rel1} becomes
	\begin{equation}
		\log\left(\frac{\sin(x)}{\cos\left(\frac\pi4\right)}\right)=\log_{\overline{\chi}}\left(\frac{\sin(x)}{\cos(x)}\right),\quad {\rm with}\quad \overline \chi(s)=s(1+s^2).
	\end{equation}
	Moreover, straightforward computations leads to 
	$$\log_{\overline{\chi}}(s)=\int_1^s \frac{\diff t}{t(1+t^2)}=\log s-\frac 12\log\left(1+s^2\right)+\frac 12\log2,$$
	that is
	$$\log_{\overline{\chi}}(\tan x)=\log (\tan x)-\frac 12\log\left(\frac{\sin^2(x)+\cos^2(x)}{\cos^2(x)}\right)+\frac 12\log2,$$
	and after some algebriac manipulations it follows that
	$$
	\cos^2(x)+\sin^2(x)=\left(\sqrt 2\cos(\pi/4)\right)^2={\rm cte}.
	$$
	Furthermore, since in the general case $\sinf(\piftwo/2)=1$  and $\cosf(\piftwo/2)=0,$  one necessarily obtains that $\cos(\pi/4)=1/\sqrt 2$ and 
	$$
	\cos^2(x)+\sin^2(x)=1.
	$$
\end{remark}

We continue by obtaining the following alternative expression for the $\chi$-pythagorean relation in Eq.~\eqref{eq:gen_pyth_rel1}. 
\begin{teorema} Let $\chi$ a positive and increasing function satisfying the conditions in Eq.~\eqref{eq:conditions}. Then, 
	\begin{equation}\label{eq:gen_pyth_rel2}
		\log_{\widetilde \chi}\left(\frac{\sinf(x)}{\cosf(x)}\right)=\log\left(\frac{\sinf\left(\frac\piffour4\right)}{\cosf(x)}\right) \, ,\quad x\in \overline{\Omega}_{\chi}\, ,
	\end{equation}
	where 
	\begin{equation}\label{eq:chi_tilde}
		\widetilde{\chi}(s)=s\left(1+\frac1{\chi^{2}(s)}\right),
	\end{equation}
	assuming that $\widetilde{\chi}$ is a positive function for $0<s\leqslant x.$ 
\end{teorema}	  
\begin{proof}
First, we note that, recalling Eqs.~\eqref{eq:cospi4} and~\eqref{eq:cosin_alt_exp}, the $\chi$-cosine function can be written as
	\begin{equation}
		\cosf(x)=\sin\left(\frac\piffour4\right)
		\exp\left(-\int_{\piffour/4}^x \frac{\chi^2(\tanf(t))\diff t}{\tanf(t)}\right),
	\end{equation}
	or equivalently,
	\begin{equation}
		\cosf(x)=\sin\left(\frac\piffour4\right)
		\exp\left(-\int_{1}^{\tanf (x)} \frac{\chi^2(s)\diff s}{s(1+\chi^2(s))}\right).
	\end{equation}
The proof trivially follows by taking into account the definition of the $\chi$-logarithm function in Eq.~\eqref{eq:chi_log}.
\end{proof}
\begin{remark}
		We now let $\chi_p(s)=s^{p/2},$ with $p>1$, for which choice the functions $\sin_\chi$ and $\cos_{\chi}$ reduce to the generalized $p$-trigonometric functions $\sin_p$ and $\cos_p$, as dicussed in Remarks~\ref{remark:q-case1} and~\ref{remark:q-case2}. Then, in order to apply the Pythagorean-like relation~\eqref{eq:gen_pyth_rel2} we first compute  $\widetilde \chi(s)=s(1+s^{-p})=\frac{s^p+1}{s^{p-1}},$ and so
	$$\log_{\widetilde\chi}(x)=\int_1^x\frac{s^{p-1}\diff s}{s^p+1}=\frac 1p \log (s^p+1)\Bigg|_1^x=\frac 1p \log (x^p+1)-\frac{\log2}{p}.$$
	From where 
	$$
	\frac 1p \log \left(\frac{\sin^p_p(x)}{\cos^p_p(x)}+1\right)-\frac{\log2}{p}=\log\left(\frac{\sin_p\left(\frac{\pi^\angle_p}4\right)}{\cos_p(x)}\right) 
	$$
	or equivalently
	$$
	\log \left(\frac{\sin^p_p(x)+\cos^p_p(x)}{\cos^p_p(x)}\right)=\log\left(2\,\frac{\sin^p_p\left(\frac{\pi^\angle_p}4\right)}{\cos^p_p(x)}\right)
	$$
	i.e.
	$$
	\sin^p_p(x)+\cos^p_p(x)=2\,\sin^p_p\left(\frac{\pi^\angle_p}4\right)={\rm cte}.$$
	In addition 
	$$
	\sin_{p}^{p}(\pi^\perp_p/2)+\cos_{p}^{p}(\pi^\perp_p/2)=1+0=1.
	$$
	
	That is, we recover the generalized pythagorian relation
	$$\sin_p^p(x)+\cos_p^p(x)=1,$$
	which includes the case $p=2$ as the classical identity
	$$\sin^2(x)+\cos^2(x)=1.$$
	In addition we also obtain 
	$$\sin_p\left(\frac{\pi^\angle_p}4\right)=\frac1{\sqrt[p]2}.$$
\end{remark}

\subsection{The $\chi$-arcsine and $\chi$-arccosine functions}
\begin{teorema} Assuming that the conditions~\eqref{eq:conditions} hold, the inverse function of the $\chi$-sine function is given by
	\begin{equation}\label{eq:arcsin}
		\arcsin_\chi(x)=\int_0^x\frac{\diff t}{\xi_{\chi}(t)}, \quad  0\leqslant x \leqslant 1\, ,
	\end{equation}
	where
	\begin{equation}\label{eq:xi_arcsin}
		\xi_{\chi}(t):=\frac{t}{\exp_{\overline\chi}\left(\log\left(\frac{t}{\mathcal K_\chi}\right)\right)},
	\end{equation}
	or equivalently,
\begin{equation}\label{eq:arcsin2}
	\arcsin_\chi(x)=\int_{-\infty}^{\log\left(\frac{x}{\mathcal K_\chi}\right)} \exp_{\overline\chi}\left(s\right)\diff s.
\end{equation}
\end{teorema}
\begin{proof}
	The $\chi$-generalized Pythagorean-like relation~\eqref{eq:gen_pyth_rel1},
\begin{equation}
	\log\left(\frac{\sinf(x)}{\cosf\left(\frac\piffour4\right)}\right)=\log_{\overline{\chi}}\left(\frac{\sinf(x)}{\cosf(x)}\right),
\end{equation}
gives
\begin{equation}\label{3.15}
	\cosf(x)=\frac{\sinf(x)}{\exp_{\overline\chi}\left(\log\left(\frac{\sinf(x)}{\mathcal K_\chi}\right)\right)},
\end{equation}
where $\mathcal K_\chi=\cosf\left(\frac\piffour4\right),$ as defined in \eqref{eq:cospi4}. This equation can be written as
\begin{equation}\label{eq:cos_xi_sin}
	\cosf(x)=\xi_\chi(\sinf(x)).
\end{equation}
	where $\xi_{\chi}$ is defined in~\eqref{eq:xi_arcsin}. From the latter equation, using the inverse function theorem we can deduce that
\begin{equation}
	\arcsin_\chi(x)=\int_0^x\frac{\diff t}{\xi_\chi(t)}=\int_0^x\frac{\diff t}{t}\exp_{\overline\chi}\left(\log\left(\frac{t}{\mathcal K_\chi}\right)\right), \quad 0\leq x \leq 1.
\end{equation}
Finally, making the change of variable $s=\log\left(\frac{t}{\mathcal K_\chi}\right)$, Eq.~\eqref{eq:arcsin2} trivially follows. 
\end{proof}

\begin{remark}

	From Eqs.~\eqref{eq:arcsin} and \eqref{eq:arcsin2} it is straightforward to see that $\arcsin_\chi(0) =0$.  We now show that $\arcsin_\chi(1) = \frac{\piftwo}{2}$. We have
		\begin{equation}\label{eq:arcsinchi11}
			\arcsin_\chi(1) = \int_0^1 \frac{\diff t}{\xi_{\chi}(t)} = \int_0^1 \frac{\exp_{\overline\chi}\left(\log\left(\frac{t}{\mathcal K_\chi}\right)\right) }{t}\, dt = \int_{-\infty}^{-\log \mathcal K_\chi} \exp_{\overline\chi}\left(s\right) \, ds \,,
		\end{equation}
		where, for the third equality we have used the change of variable $s=\log\left(\frac{t}{\mathcal K_\chi}\right)$. Now, using the fact that $\exp_{\overline\chi} = \log_{\overline\chi}^{-1}$ and the inverse function theorem, we have that
		\[
		\exp_{\overline\chi}'(s) = \frac{1}{\log_{\overline \chi}'(\exp_{\overline\chi}(s)) } = \overline \chi(\exp_{\overline\chi}(s)) \, .
		\]
		Then, making the following change of variable in the last integral in \eqref{eq:arcsinchi11},
		\begin{align*}
			& u = \exp_{\overline\chi}(s), \qquad du = \exp_{\overline\chi}'(s)\, ds = \overline \chi(\exp_{\overline\chi}(s))\,ds, \quad ds = \frac{1}{\overline\chi(u)}\, du\, , \\
			& u = \exp_{\overline\chi}(-\infty) = 0, \qquad u = \exp_{\overline\chi}(-\log \mathcal K_\chi) = \exp_{\overline\chi}\left(\int_{1}^{\infty} \frac{1}{s(1+\chi^2(s))}\, ds \right) = \exp_{\overline\chi}\left(\log_{\overline\chi}(\infty) \right) = \infty\, ,
		\end{align*}
		we see that
		\begin{equation}\label{eq:arcsinchi12}
			\arcsin_\chi(1) = \int_{0}^{\infty} \frac{u}{\overline\chi(u)}\, du = \int_{0}^{\infty} \frac{1}{1+\chi^2(u)}\, du = \arctanf(\infty) = \frac{\piftwo}{2}\, .
		\end{equation}

\end{remark}

\begin{teorema} Assuming that the conditions~\eqref{eq:conditions} hold, the inverse function of the $\chi$-cosine function is given by
\begin{equation}\label{eq:arccos}
	\arccos_\chi(x)=\frac\piftwo2+\int_\infty^{\log(\mathcal K_\chi/x)} \frac{\exp_{\widetilde \chi}(s) }{\chi^2\left(\exp_{\widetilde \chi}(s)\right)}\diff s, \quad 0\leqslant x \leqslant \cosf(0)
\end{equation}
where $\widetilde{\chi}$ is defined in Eq.~\eqref{eq:chi_tilde}, provided that $\piftwo<\infty$. Alternatively,
\begin{equation}\label{eq:arccos2}
	\arccos_\chi(x)=\frac\piftwo2+\int_\infty^{\log(\mathcal K_\chi/x)} \left(\widetilde\chi\left(\exp_{\widetilde \chi}(s)\right)-\exp_{\widetilde \chi}(s)\right)\diff s\, .
\end{equation}
\end{teorema}
\begin{proof}
	From Eq.~\eqref{eq:gen_pyth_rel2} it follows that
\begin{equation}
	\sinf(x)=\cosf(x)\,\exp_{\widetilde \chi}\left(\log\left(\frac{\sinf\left(\frac\piffour4\right)}{\cosf(x)}\right)\right).
\end{equation}
Then using Eq.~\eqref{eq:cosf_der},
\begin{eqnarray}
	\cosf' (x)&=&-\frac{\cosf^2(x)}{\sinf(x)}\chi^2\left(\tanf(x)\right)
\\
&=&-\frac{\cosf(x)}{\exp_{\widetilde \chi}\left(log \mathcal K_\chi -\log(\cosf(x))\right)}\chi^2\left(\exp_{\widetilde \chi}\left(\log \mathcal K_\chi-\log(\cosf(x))\right)\right)
\end{eqnarray}
That is, we can write

\begin{equation}
	\cosf' (x)=\widetilde \Xi_\chi (\cosf(x)),
\end{equation}
with 
\begin{equation}
	\widetilde{\Xi}_\chi(z)=-\left(\frac{\chi^2\left(\exp_{\widetilde \chi}\left(\log(\mathcal K_\chi/z)\right)\right)}{\exp_{\widetilde \chi}\left(\log(\mathcal K_\chi/z)\right)}\right) z.
\end{equation}
From the latter equation, and the fact that $\cos(\piftwo/2)=0$, it follows that
\begin{equation}
	\arccos_\chi(x)=\frac\piftwo2+\int_0^x\frac{\diff z}{\widetilde{\Xi}_\chi(z)}=\frac\piftwo2-\int_0^x\frac{\diff z}{z}  \frac{\exp_{\widetilde \chi}\left(\log(\mathcal K_\chi/z)\right)}{\chi^2\left(\exp_{\widetilde \chi}\left(\log(\mathcal K_\chi/z)\right)\right)}.
\end{equation}
Now we apply the change of variable $s=\log\frac{\mathcal K_\chi}{z}$, $\diff s=-\frac{\diff z}z$, which leads to
\begin{equation}
	\arccos_\chi(x)=\frac\piftwo2+\int_\infty^{\log\frac{\mathcal K_\chi}{x}} \frac{\exp_{\widetilde \chi}(s) }{\chi^2\left(\exp_{\widetilde \chi}(s)\right)}\diff s 
	.
\end{equation}

Finally, using~\eqref{eq:chi_tilde} one can check that

\begin{equation}
	\frac{\exp_{\widetilde \chi}(s)}{\chi^2\left(\exp_{\widetilde \chi}(s)\right)}=\widetilde\chi\left(\exp_{\widetilde \chi}(s)\right)-\exp_{\widetilde \chi}(s).
\end{equation}
\end{proof}
\begin{remark}
	We now show that $\arccos_\chi(\cosf(0)) = 0$. We begin by noting that
		\begin{equation}\label{eq:arccos11}
			\hspace{-0.95cm}	\arccos_\chi(\cosf(0)) = \frac\piftwo2+\int_\infty^{\log(\mathcal K_\chi/\cosf(0))} \left(\widetilde\chi\left(\exp_{\widetilde \chi}(s)\right)-\exp_{\widetilde \chi}(s)\right)\diff s.
		\end{equation}
		Then, using the change of variable $u = \exp_{\overline\chi}(s)$ in the last integral in \eqref{eq:arccos11} and 
		\[
		\log \mathcal K_{\chi} - \log \cos_{\chi}(0) = - \int_{1}^{\infty} \frac{1}{s(1+\chi^2(s))}\, ds - \int_{0}^{1} \frac{\chi^2(s)}{s(1+\chi^2(s))}\, ds + \int_{1}^{\infty} \frac{1}{s(1+\chi^2(s))}\, ds = \int_{1}^{0} \frac{1}{\widetilde \chi(s)}\, ds = \log_{\widetilde \chi }( 0) \, ,
		\]
		we obtain
		\begin{equation}\label{eq:arccos12}
			\arccos_\chi(\cosf(0))=\frac\piftwo2+ \int_{\infty}^{0} \frac{\widetilde\chi (u)-u}{\widetilde\chi(u)}\, du = \frac\piftwo2- \int_{0}^{\infty} \frac{1}{1+\chi^2(u)}\, du =  \frac\piftwo2-\arctanf(\infty) =  \frac\piftwo2- \frac\piftwo2 = 0\, .
	\end{equation}
\end{remark}

\begin{remark}[The $\chi$-arcsine and $\chi$-arccosine functions vs the $\chi$-arctan function] We deduce now alternative expresions for the $\chi$-arcsine and $\chi$-arccosine functions by also involving the $\chi$-arctan function.

\begin{itemize}[]
\item \textbf{The $\chi$-arcsine  and $\chi$-arctan functions}
\end{itemize}

We start now with Eq.~\eqref{eq:cos_xi_sin}, 

\begin{equation}
	\cosf(\arcsin_\chi(x))=\xi_\chi(x),
\end{equation}
where $\xi_\chi$ is defined in Eq.~\eqref{eq:xi_arcsin}. Then
\begin{equation}
	\tanf(\arcsin_\chi(x))=\frac{\sinf(\arcsin_\chi(x))}{\cosf(\arcsin_\chi(x))}=\frac{x}{\xi_\chi(x)}=\exp_{\overline\chi}\left(\log\left(\frac{x}{\mathcal K_\chi}\right)\right).
\end{equation}
Finally, we have
\begin{equation}
	\arcsin_\chi(x)=\arctanf\left(\exp_{\overline\chi}\left(\log\left(\frac{x}{\mathcal K_\chi}\right)\right)\right)\, ,  \quad 0 \leqslant x \leqslant 1. 
\end{equation}

\begin{itemize}
\item \textbf{The $\chi$-arccosine and the $\chi$-arctan functions}
\end{itemize}	

	From Eq.~\eqref{eq:gen_pyth_rel2} it follows that 
\begin{equation}
	\sinf(x)=\cosf(x)\,\exp_{\widetilde \chi}\left(\log\left(\frac{\sinf\left(\frac\piffour4\right)}{\cosf(x)}\right)\right).
\end{equation}
That is, 
\begin{equation}
	\sinf(\arccos_\chi(x))=\Xi_\chi(x),
\end{equation}
where
\begin{equation}
	\Xi_\chi(x)=x\,\exp_{\widetilde \chi}\left(\log\left(\mathcal K_\chi/x\right)\right),\quad \mathcal K_\chi=\sinf\left(\frac\piffour4\right).
\end{equation}
Consequently,
\begin{equation}
	\tanf(\arccos_\chi(x))=\frac{\sinf(\arccos_\chi(x))}{\cosf(\arccos_\chi(x))}=\frac{\Xi_\chi(x)}x=\exp_{\widetilde \chi}\left(\log \left(\mathcal K_\chi/x\right)\right).
\end{equation}
We then obtain the following expression for  the $\chi$-arccos function
\begin{equation}\label{eq:arccosdef2}
	\arccos_\chi(x)=\arctanf\left(\exp_{\widetilde \chi}\left(\log \left(\frac{\mathcal K_\chi}x\right)\right)\right), \quad 0 \leqslant x \leqslant \cosf(0) .
	\end{equation}

\end{remark}

	\section{Some illustrative examples}

In this section we  give some examples   of the generalized trigonometric functions defined above. We obtain the particular perpendicular and oblique $\pi_\chi$ constants as well as the correspoding Pythagorean relations.
	
	\subsection{The case  $\chi(s):=\chi_{a,b}(s)=\sqrt{as^2+b}$}\label{ssec:ej1}

		As a useful toy example, we consider the function $\chi_{a,b}(s)=\sqrt{as^2+b}$ with $a,b>0$. We have that, for any $x\in[0,\infty)$,
		$$\arctan_{\chi_{a,b}}(x)=\int_0^x\frac{\diff s}{1+b+as^2}.$$
		That is, 
		$$\arctan_{\chi_{a,b}}(x)=\frac1{\sqrt{(1+b)a}} \arctan \left(\sqrt{\frac{a}{1+b}}\,x\right),$$
		and consequently,
		\begin{equation}\label{eq:tan_special_case}
			\tan_{\chi_{a,b}}(x)=\sqrt{\frac{1+b}a} \tan\left(\sqrt{(1+b)a}\,x\right)\, .
		\end{equation}
		In this case, the corresponding generalized  $\pi_\chi$ constants  are given by
		\begin{align*}
			\pi_{\chi_{a,b}}^\perp & =  2\arctan_{\chi_{a,b}}(\infty) 
			= \frac{\pi}{\sqrt{(1+b)a}} \, , \\[0.5em]
			\pi_{\chi_{a,b}}^\angle & = 4\arctan_{\chi_{a,b}}(1) = \frac4{\sqrt{(1+b)a}} \arctan \left(\sqrt{\frac{a}{1+b}}\right).
		\end{align*}
			
			 Note   that $\chi_{a,b}(-s) = \chi_{a,b}(s)$ and $\chi_{a,b}(s) \sim s$ as $s\to \infty$, which means that both $\arctan_{\chi_{a,b}}$ and $\tan_{\chi_{a,b}}$ are naturally odd

 Consequently, the $\sinf$ and $\cosf$ functions are given by
\begin{eqnarray}
\sin_{\chi_{a,b}}(x)
&=&\exp\left(\sqrt{\frac a{1+b}} \int_{ \frac{\pi}{2\sqrt{(1+b)a}}}^x\frac{\diff s}{\tan\left(\sqrt{(1+b)a}\,s\right)}\right)
\\
&=&\exp\left(\frac 1{1+b} \log\left(\sin\left(\sqrt{a(b+1)}\,x\right)\right)\right)=\left(\sin\left(\sqrt{a(b+1)}\,x\right)\right)^{\frac 1{1+b}}
\end{eqnarray}

and 
$$
\cos_{\chi_{a,b}}(x)=\sqrt{\frac a{b+1}}\frac{\cos\left(\sqrt{a(b+1)}\,x\right)}{\left(\sin\left(\sqrt{a(b+1)}\,x\right)\right)^{\frac {b}{1+b}}}\, .
$$

As can be observed, since $\chi(0)\neq 0$, $\cos_{\chi_{a,b}}(0) =\infty$, as discussed  following  conditions~\eqref{eq:conditions}. On the other hand, if $b=0$, i.e., $\chi_{a,0}(s) = \sqrt{a}s$ and, thus, $\chi_{a,0}(0)=0$, then,
\begin{equation}\label{eq:tan_special_case_b_0}
	\tan_{\chi_{a,0}}(x)=\sqrt{\frac{1}{a}} \tan\left(\sqrt{a}\,x\right) , \quad \sin_{\chi_{a,0}}(x) = \sin\left(\sqrt{a}\,x\right), \quad \text{and} \quad  \cos_{\chi_{a,0}}(x) = \sqrt{a}\cos(\sqrt{a}\, x)
\end{equation}
where $\cos_{\chi_{a,0}}(0) = \sqrt{a}\cos(\sqrt{a}\,0) = \sqrt{a}$. For completeness,
\begin{equation*}
	\pi_{\chi_{a,0}}^\perp  = \frac{\pi}{\sqrt{a}} \, , \qquad \text{and} \qquad \pi_{\chi_{a,0}}^\angle  = \frac4{\sqrt{a}} \arctan \left(\sqrt{a}\right).
\end{equation*}

	We now derive the Pythagorean-like relation. First, we compute the expression for $\log_{\overline \chi}(x)$ with 
$$\overline \chi(s)=s(1+\chi_{a,b}^2(s))=ks(1+cs^2),\quad{\rm where}\quad k=b+1,\quad {\rm and}\quad  c=\frac a{b+1}. $$
We apply the definition~\eqref{eq:chi_log}, and, after some algebraic manipulations, we obtain
$$\log_{\overline \chi}(x)=\frac1k \left[\log \frac{x}{\sqrt{1+cx^2}}+\log\sqrt{1+c}\right].$$
Now, the Pythagorean relation in~\eqref{eq:gen_pyth_rel1} reads
$$\log\left(\frac{\sinfab(x)}{\mathcal K_{\chi_{a,b}}}\right)=\frac1k \left[\log \frac{\tanfab(x)}{\sqrt{1+c\tanfab^2(x)}}+\log\sqrt{1+c}\right],$$
where $\mathcal K_{\chi_{a,b}}$ denotes the  corresponding  constant defined in~\eqref{eq:cospi4}. After further algebraic manipulations,
we obtain that this Pythagorean relation may be expressed
$$ \sinfab^{2b}(x)\left[(b+1)\cosfab^2(x)+a\sinfab^2(x)\right]=a.$$
This identity is readily seen to hold.

\subsection{The case $\chi(s)=\sqrt{s^2+2s}$}\label{ssec:ej2}

Choosing $\chi(s)=\sqrt{2s+s^2}$, we obtain
$$
\arctan_{\chi}(x)=\int_0^{x}\frac{\diff s}{1+2s+s^2}=\int_0^{x}\frac{\diff s}{(s+1)^2}
=\frac x{x+1},
$$
and then
$$\tanf(x)=\frac{1}{1-x}-1=\frac{x}{1-x}.$$

The values of the $\pi_\chi$ constants are 
$$\piftwo=2\arctanf(\infty)=2,\quad {\rm and}\quad \piffour=4\arctanf(1)=2.$$

The $\chi$-sine  and $\chi$-cosine functions are given by
\begin{equation}\label{eq:sin_chi_quad}
\sinf(x)=\exp\left(\int_1^x \frac{1-s}{s}\diff s\right)=\exp\left(\log x -x+1\right) = x e^{1-x}\, ,
\end{equation}
and
\begin{equation}\label{eq:cos_chi_quad} 
\cosf(x)=\frac{\sinf(x)}{\tanf (x)}=(1-x)e^{1-x},
\end{equation}
respectively. In addition, we have that

$$\sinf(0)=0,\quad \cosf(0)=e\, ,$$
$$\sinf(\piffour/4)=\cosf(\piffour/4)=\cosf(1/2)=\frac{\sqrt e}{2}\, ,$$
and 
$$\sinf(\piftwo/2)=\sinf(1)=1,\quad \cosf(\piftwo/2)=\cosf(1)=0\, .$$
From~\eqref{eq:sin_chi_quad} and \eqref{eq:cos_chi_quad}, we now deduce the corresponding Pythagorean-like relation. To do so, we start by computing the associated $\overline \chi $ function 
$$\overline \chi(s)=s(1+\chi^2(s))=s(s+1)^2.$$
Straightforward computations leads to
$$
\log_{\overline \chi}(x)
=\log\left(\frac{2x}{x+1}\right) +\frac{1-x}{2(x+1)}
$$
The pythagorean relation 
$$
\log\left(\frac{\sinf(x)}{\cosf(\piffour/4)}\right)=\log_{\overline \chi}\left(\frac{\sinf(x)}{\cosf(x)}\right)
$$
becomes
$$
\log\left(\frac{\sinf(x)}{\sqrt e/2}\right)=\log_{\overline \chi}\left(\tanf(x)\right)=\log\left(\frac{2\tanf(x)}{\tanf(x)+1}\right)+\frac{1-\tanf(x)}{2(\tanf(x)+1)}
$$
or equivalently
$$
\log\left(\frac{\sinf(x)}{\sqrt e/2}\right)
=\log\left(\frac{2\sinf(x)}{\sinf(x)+\cosf(x)}\right)+\frac{\cosf(x)-\sinf(x)}{2(\sinf(x)+\cosf(x))}
$$
That is,

$$
2\log\left(\sinf(x)+\cosf(x)\right)+\frac{\sinf(x)-\cosf(x)}{\sinf(x)+\cosf(x)}=1.
$$
One can easily check this equality by noting that $\sinf(x)+\cosf(x)=e^{1-x}$ and $\sinf(x)-\cosf(x)=(2x-1)e^{1-x}.$

\subsection{The case $\chi(s)=\sqrt{\frac{s^2+3s}2}$}\label{ssec:ej3}
We now consider the case $\chi(s)=\sqrt{\frac{s^2+3s}2}$. We will make use of this particular case in Section 6.1.

By definition, we have that
$$\arctanf(x)=2\int_0^x\frac{\diff s}{s^2+3s+2}=2\left(\log(s+1)-\log(s+2)\right)\Big|_0^x=2\left(\log\frac{x+1}{x+2}+\log 2\right).$$
Consequently, letting $y=\arctanf(x)$, we find
$$
x=\frac{2e^{y/2}-2}{2-e^{y/2}}.
$$
That is,
$$\tanf(x)=\frac{2e^{x/2}-2}{2-e^{x/2}},\quad x\in (0,\piftwo/2),$$
where 
$$\frac\piftwo 2=\arctanf(\infty)=2\log 2,$$
and in addition
$$\frac\piffour 4=\arctanf(1)=2\log\left(\frac43\right).$$

We now compute the $\chi$-sine function as
\begin{eqnarray}
	\sinf(x)&=&\exp\left(\int_{\piftwo/2}^x \frac{2-e^{s/2}}{2e^{s/2}-2}\diff s\right)
	\\
	&=& \exp\left(\log|2e^{s/2}-2|-s\Bigg|_{2\log2}^x\right)=4e^{-x}(e^{x/2}-1).
\end{eqnarray}
Then, the $\chi$-cosine function is given by
\begin{equation}
	\cosf(x)=\frac{\sinf(x)}{\tanf(x)}=\frac{4e^{-x}(e^{x/2}-1)}{2e^{x/2}-2}(2-e^{x/2})=2e^{-x}(2-e^{x/2}).
\end{equation}
For completeness, we also derive $\cosf$ defined as the derivative of the $\chi$-sine function
$$
	\cosf(x) = \sinf'(x)= -4 e^{-x} (e^{x/2}-1)+2e^{-x}e^{x/2}=2e^{-x}(2-e^{x/2})\, .
	$$
We now compute the $\overline \chi$-logarithm,  in order  to write the Pythagorean-like relation. 
First, we have
$$
\overline\chi(s)=s(1+\chi^2(s))=s\left(1+\frac{s^2+3s}2\right)=\frac{s(s+1)(s+2)}2
$$
and, after some algebraic manipulations, it follows that
$$\log_{\overline{\chi}}(x)=2\int_1^x \frac{\diff s}{s(s+1)(s+2)}=\log\left(\frac{|s(s+2)|}{(s+1)^2}\right)\Big|_1^x=\log\left(\frac{|x(x+2)|}{(x+1)^2}\right)-\log\left(\frac 34\right)$$
Now, using the fact that $\sinf(\piffour/4)=4\frac{9}{16}(\frac43-1)=\frac3{4},$ and also Eq.~\eqref{eq:gen_pyth_rel1}, it follows that
$$
\log\left(\frac{4\sinf(x)}3\right)=\log\left(\frac{4\tanf(x)(\tanf(x)+2)}{3(\tanf(x)+1)^2}\right),
$$
or, equivalently, 
$$
(\sinf(x)+\cosf(x))^2=\sinf(x)+2\cosf(x).
$$

	\subsection{The case $\chi (s)=\sqrt{e^s-1}$} \label{ssec:ej4}
We start by computing the $\chi$-$\arctan$ function. For any $x>0$, 
$$\arctan_{\chi}(x)=\int_0^x \frac{\diff s}{e^{s}}
=\frac{e^x-1}{e^x}.$$ 
Then, we obtain the inverse function
$$y=1-e^{-x}\Longrightarrow x=-\log(1-y)\, .$$
That is, for any $x\in[0,1)$
$$\tanf(x)=\log\left(\frac1{1-x}\right).$$
Moreover,
$$\frac{\piffour}4=\arctanf(1)=\frac{e-1}{e},\qquad \frac{\piftwo}2 =\arctanf(\infty)=1\, .$$
Consequently,
$$\sinf(x)=\exp\left(-\int_1^x\frac{\diff s}{\log (1-s)}\right)
=\exp\left(\int_0^{1-x}\frac{\diff t}{\log (t)}\right)
=e^{{\rm li}(1-x)}\, ,
$$
where $\text{li}$ denotes the logarithmic integral function~\cite{Gradshteyn_Tables}.
The $\chi$-cosine function is then given by 
$$\cosf(x)=\frac{\sinf(x)}{\tanf(x)}=\frac{e^{{\rm li}(1-x)}}{\log\left(\frac1{1-x}\right)}\, .$$

The relevant values are given by (for  $\cosf(0)$  see Proposition 1)
\begin{align*}
&\sinf(0)=0,\quad \cosf(0)=\exp\left(\int_0^1\frac{1-e^{-s}}{s}\diff s +{\rm Ei}(-1)\right), \\[0.5em]
&\sinf(\piffour/4)=\cosf(\piffour/4)=\cosf\left(\frac{e-1}{e}\right)=e^{{\rm li} (1/e)}, \\[0.5em]
&\sinf(\piftwo/2)=\sinf(1)=e^{{\rm li}(0)}=1,\quad \cosf(\piftwo/2)=\cosf(1)=0\, ,
\end{align*}
where $\text{Ei}$ denotes the exponential integral function~\cite{Gradshteyn_Tables}.

\noindent
For the corresponding Pythagorean-like identity we now compute the associated function
$$\overline \chi(s)=s(1+\chi^2(s))=se^s.$$
Then,
\begin{align*}
\log_{\overline \chi}(x)&=\int_1^x \frac{e^{-s}}{s}\diff s = \int_1^{\infty} \frac{e^{-s}}{s}\diff s - \int_x^{\infty} \frac{e^{-s}}{s}\diff s 
\\
&={\rm Ei}(-x)-{\rm Ei}(-1).
\end{align*}
As a consequence, 
$$
\log\left(\frac{\sinf(x)}{\cosf(\piffour/4)}\right)={\rm Ei}\left(-\frac{\sinf(x)}{\cosf(x)}\right)-{\rm Ei}(-1)
$$
or, equivalently,
$$
\log\left(\sinf(x)\right)-{\rm Ei}\left(-\frac{\sinf(x)}{\cosf(x)}\right)={\rm li}(1/e)-{\rm Ei}(-1)\, .
$$

That is,
$$
{\rm li}\left(1-x\right)-{\rm Ei}\left(\log(1-x)\right)={\rm li}(1/e)-{\rm Ei}(-1)\, , 
$$
 which, since the right-hand-side is zero, 
recovers the fundamental relation between the exponential and logarithmic integral functions,
$${\rm Ei}(\log x)={\rm li}(x).$$

\section{Generalized Beta and Ateb functions and non-linear oscillators}

In this section we  give some of the main results  of the paper. More precisely, we first use the previous definitions and results to define generalized versions of the incomplete Beta and the Ateb sine functions,  and give particular examples. We then  prove that the exact  solutions of broad  subfamilies of \eqref{eq:nonlinear_osc_f} can be written in closed-form by  means of these new incomplete $\chi$-Beta and $\chi$-Ateb functions. An important tool that we will use to accomplish this task is the family of generalized Sundman-like transformations recently introduced by the authors in~\cite{Gordoa2026}, and whose definition will be  recalled later  in the paper. 

\subsection{Dealing with generalized Beta and Ateb functions}

	In this section we introduce an extension of the well-known Beta function, and its incomplete version. We begin by recalling the following expression for the incomplete Beta function,
	
	\begin{equation}\label{eq:Beta}
		{\rm B}(c,d;\sin^2(x))=2\int_0^x \sin^{2c-1}(t) \cos^{2d-1}(t)\diff t, \quad c,d>0, \quad x\in \left[0,\frac{\pi}{2}\right]
	\end{equation} 
	which is equivalent to the definition given in \eqref{eq:Beta0} and again reduces to the Beta function for $x=\frac{\pi}{2}$.
	The characterization of this function in~\eqref{eq:Beta} inspires the following definition.
	\begin{definition}[ $\chi$-Beta functions ]
		Let  $c$, $d$ be positive real numbers and let $\chi$ be a function satisfying the conditions in Eq.~\eqref{eq:conditions}. We define the incomplete $\chi$-Beta function as
		\begin{equation}\label{eq:Betagen}
			\Betaf\left(c,d;\sinf^2 x\right):= 2\int_0^x \sinf^{2c-1}(t)\cosf^{2d-1}(t)\diff t , \quad x\in \left[0,\frac{\pif}{2}\right],
		\end{equation}
which for $x=\frac{\pif}{2}$ reduces to the complete $\chi$-Beta function,
\[
\Betaf\left(c,d;\sinf^2 \left(\frac{\pif}{2}\right)\right)\equiv  \Betaf\left(c,d\right) = 2\int_0^{\frac{\pif}{2}} \sinf^{2c-1}(t)\cosf^{2d-1}(t)\diff t \, .
\]

	\end{definition}

	\begin{remark}
	 The $\chi$-Pythagorean relation~\eqref{eq:gen_pyth_rel1}, which allows us to express the $\chi$-cosine function in terms of the $\chi$-sine as
		\begin{equation}
			\cosf(x)=\frac{\sinf(x)}{\exp_{\overline \chi}\left(\log\left(\sinf(x)/\mathcal K_\chi\right)\right)}\, 
		\end{equation}
		(see \eqref{3.15}),
	 together with the change of variable $s=\sinf^2(t)$ in Eq.~\eqref{eq:Betagen}, lead to
			\begin{equation}
			{\rm B}_\chi\left(c,d;\sinf^2(x)\right)=\int_0^{\sinf^2(x)}s^{c-1} \left(\frac{s^\frac 12}{\exp_{\overline \chi}\left(\log\frac{ s^{\frac{1}{2}} }{\mathcal K_\chi}\right)}\right)^{2d-2}\diff s
		\end{equation}
		or, equivalently,
		\begin{equation}
			{\rm B}_\chi\left(c,d;x\right)=\int_0^{x}s^{c+d-2} \left(\exp_{\overline \chi}\left(\log\frac{s^{\frac{1}{2}}}{\mathcal K_\chi}\right)\right)^{2-2d}\diff s, \quad x\in [0,1]\, ,
		\end{equation}
		which is an alternative definition of the incomplete $\chi$-Beta function. As before, for $x=1$ we obtain the following alternative definition of the  complete  $\chi$-Beta function,
\[
\Betaf\left(c,d;1\right)\equiv  \Betaf\left(c,d\right) = \int_0^{1}s^{c+d-2} \left(\exp_{\overline \chi}\left(\log\frac{s^{\frac{1}{2}} }{\mathcal K_\chi}\right)\right)^{2-2d}\diff s\, .
\]		

	\end{remark}

\begin{definition}[$\chi$-Ateb sine function]
	Let $\alpha$ be a real number, and let $\chi$ be a positive increasing function. Then, we define the $\chi$-Ateb sine function as

\begin{equation}
	{\rm sa}_\chi(n,m;x)=\left[{\rm B}_\chi^{-1}\left(\frac1{n+1},\,\frac 1{m+1};2x\right)\right]^{\frac1{n+1}}, \quad x\in \left[0,\frac{\pif}{2}\right],	
\end{equation}
and for $m=1$,
\begin{equation}
	{\rm sa}_\chi(n;x)=\left[{\rm B}_\chi^{-1}\left(\frac1{n+1},\,\frac 12;2x\right)\right]^{\frac1{n+1}}, \quad x\in \left[0,\frac{\pif}{2}\right].	
\end{equation}
\end{definition}

\subsection*{Examples}

 We now return to Examples 4.1-4.4. 
	
	\subsubsection*{ 4.1  The case $\chi(s)= \chi_{a,b}(s)=  \sqrt{as^2+b}$}
	Regarding the example in Section~\ref{ssec:ej1}, the  corresponding  incomplete incomplete $\chi$-Beta function is defined as the function satisfying
\begin{eqnarray*}
	{\rm B}_{\chi_{a,b}} \left(c,d;\sin^{\frac2{1+b}}\left(\sqrt{a(b+1)}\,x\right)\right)&=& 2\int_0^x \sin^{\frac{2c-1}{1+b}}\left(\sqrt{a(b+1)}\,s\right)\left(\sqrt{\frac{a}{b+1}}\frac{\cos\left(\sqrt{a(b+1)}\,x\right)}{\sin^{\frac b{1+b}}\left(\sqrt{a(b+1)}\,x\right)}\right)^{2d-1}\diff s
	\\
	&=&2 \frac{a^{d-1}}{(b+1)^d}
	\int_0^{\sqrt{a(b+1)}\,x} \sin^{\frac{b+2c-2bd-1}{1+b}}\left(t\right)\,\cos^{2d-1}\left(t\right)\diff t
		\\
	&=& \frac{a^{d-1}}{(b+1)^d}	{\rm B}\left(\frac{b+c-2bd}{b+1},d; \sin^2(\sqrt{a(b+1)}\,x)\right).
\end{eqnarray*}
Equivalently,
\begin{equation}\label{eq:chi_a-b_beta}
	{\rm B}_{\chi_{a,b}} \left(c,d;x\right)= \frac{a^{d-1}}{(b+1)^d}	{\rm B}\left(\frac{b+c-2bd}{b+1},d; x^{ 1+b }
	\right).
\end{equation}

	\subsubsection*{ 4.2  The case $\chi(s)=\sqrt{s^2+2s}$}

In this case the application of the definition of the incomplete $\chi$-Beta function gives,
\begin{eqnarray}
	{\rm B}_\chi\left(c,d;x^2e^{2-2x}\right)&=&2\int_0^x \left(s e^{1-s}\right)^{2c-1}  \left((1-s)e^{1-s}\right)^{2d-1}\diff s
	\nonumber \\&=&
	2\int_0^x s^{2c-1}  (1-s)^{2d-1}e^{2(1-s)(c+d-1)}\diff s
\end{eqnarray}
Equivalently, using the Lambert $W$ function we can write

\begin{eqnarray}
	{\rm B}_\chi\left(c,d;x\right)&=&
	2\int_0^{-W\left(-\frac{\sqrt x}e\right)} s^{2c-1}  (1-s)^{2d-1}e^{2(1-s)(c+d-1)}\diff s\, ,
\end{eqnarray}
which does not seem to be related to the classical incomplete Beta function, in a simple way at least. 

	\subsubsection*{ 4.3  The case $\chi(s)=\sqrt{\frac{s^2+3s}{2}}$}

In this case the application of the definition of the incomplete $\chi$-Beta function gives

 Note changes to numbers 

\begin{eqnarray}
	{\rm B}_\chi\left(c,d;16 e^{-2x}(e^{x/2}-1)^2\right)&=&2\int_0^x \left(4e^{-s}(e^{s/2}-1)\right)^{2c-1}  \left(2e^{-s}(2-e^{s/2})\right)^{2d-1} \diff s
	\\&=&
	2^{4c+2d-2}\int_0^x (e^{s/2}-1)^{2c-1}  (2-e^{s/2})^{2d-1}e^{-2s(c+d-1)}\diff s\, .
\end{eqnarray}

Equivalently, the application of the change of variable $u =e^{\frac{s}{2}}$ leads to
\begin{equation}\label{ast}
{\rm B}_\chi\left(c,d;16 e^{-2x}(e^{x/2}-1)^2\right) = 2^{4c+4d-2}\int_1^{e^{\frac{x}{2}}} (u-1)^{2c-1}  \left(1-\frac{u}{2}\right)^{2d-1}u^{3-4(c+d)}\diff u\, ,
\end{equation}
which again does not seem to be related to the classical incomplete Beta function, in a simple manner at least.

	\subsubsection*{ 4.4  The case $\chi(s)=\sqrt{e^s-1}$}
In the case of the last example the incomplete $\chi$-Beta function is given by 
\begin{equation}
	\Betaf\left(c,d;e^{2 {\rm li}(1-x)}\right)=2\int_0^x e^{2(c+d-1){\rm li}(1-s)} \left(\log\left(\frac1{1-s}\right)\right)^{1-2d}\diff s \, ,
\end{equation}
or, equivalently,
\[
\Betaf\left(c,d;x\right)=2\int_0^{1-{\rm li}^{-1}\left(\log \sqrt{x} \right) } e^{2(c+d-1){\rm li}(1-s)} \left(\log\left(\frac1{1-s}\right)\right)^{1-2d}\diff s \, .
\]

\subsection{Applications to autonomous second order differential equations}

 We give now the definitions of the direct and inverse generalized Sundman transformations presented in~\cite{Gordoa2026}. While the inverse transformation plays a relevant role in the proof of the central theorem of this work, the direct transformation is given here for completeness. 
	\begin{definition}[Direct power-type generalized Sundman transformation (GST)]\label{def:GSTs} 
		Let $\alpha,\beta,a,b\in\mathbb{R}$. We define the transformed function $Y(s)$ of a differentiable function $y=y(x)$ as

		\begin{equation}\label{eq:GSTs1}
			Y(s) =\mathcal S_{\alpha,\beta}^{a,b}[y](s)\equiv \,
			y^\alpha(x)y'^\beta (x)
		\end{equation}

		where

		\begin{equation}\label{eq:GSTs2}
			s'(x)=y^a(x)y'^b(x),
		\end{equation}
		and where $'$ denotes a derivative with respect to $x$. $Y(s)=\mathcal S_{\alpha,\beta}^{a,b}[y](s)$, or $Y=\mathcal S_{\alpha,\beta}^{a,b}[y]$, will be called the \textit{direct generalized Sundman transformation (or transform)} of the function $y$.
	\end{definition}
	
	 Instead of giving  the definition of the inverse Sundman-like transformation in the general case,  for simplicity, we restrict ourselves here to the specific case of relevance to  this paper. The general case and  various  mathematical properties can be found in the reference~\cite{Gordoa2026},  where its usefulness  in studying certain nonlinear versions of the Schrödinger equation  is also discussed. 
	
	\begin{definition}[Inverse power-type generalized Sundman transformation (IGST)]\label{def:invGSTs}
		Let $\alpha,\beta,a,b\in\mathbb R$, $\beta\neq 0$ such that $\beta(a+1)=\alpha(b-1)$. We define the inverse generalized Sundman transformation $\widetilde{\mathcal S}_{\alpha,\beta}^{a,b}$ of a function $Y=Y(s)$, denoted $\widetilde{\mathcal S}_{\alpha,\beta}^{a,b}[Y],$ or $\widetilde{\mathcal S}_{\alpha,\beta}^{a,b}[Y](x)$ when written as a function of the variable $x$, as

		\begin{equation}\label{eq:invGSTs1}
			y(x)=\,
			\widetilde{\mathcal S}_{\alpha,\beta}^{a,b}[Y](x)
			\,\equiv\,
			\mathcal G(s),
		\end{equation}

		where
		 the function $\mathcal G(s)$ is defined as 
		
		%
%
		
		%
		\begin{equation}\label{eq:invGSTs4}
			\mathcal G(s)=K_3\,\exp\left(\int Y(s)^{\frac{1-b}\beta}\diff s\right),
		\end{equation}
		%
		
and the change of variable $s=s(x)$ is given by the inverse function of 		
		\begin{equation}\label{eq:invGSTs2}
			x(s)=  K_2+\int \frac{\diff s}{\mathcal G(s)^{a-\frac {\alpha b}\beta}Y(s)^{\frac b\beta}},
		\end{equation}

		with 
		 $K_2$ and $K_3$ in the above formulae being arbitrary constants.

	\end{definition}  
	
	Let us now continue by  recalling Corollary 5 in~\cite{Gordoa2026},  which expresses the solution of a certain class of second order differential equation via  the generalized inverse Sundman transformation.

\begin{coro} \label{coro:EDO2}
	Given $\gamma \in \mathbb{R}$ and an arbitrary function $g$, then the general solution of the family of autonomous second order differential equations 
	\begin{equation}\label{eq:EDO2}
		\frac{yy''}{y'^2}=g(y^\gamma\, y')
	\end{equation}
	is given by 
	\begin{equation}\label{eq:EDO2_sol}
		y(x)=\widetilde {\mathcal S}_{\gamma\beta,\beta}^{a,b}[Y](x),\qquad s=\int\frac{Y^{\frac{a+1}{\beta\gamma}-1}}{\beta\gamma+ \beta g\left(Y^{1/\beta}\right)}dY ,
	\end{equation}
	with $b=1+\frac{a+1}{\gamma},$ and for arbitrary values of the parameters $a$ and $\beta$. Letting $\beta=1$ and choosing $a=\gamma-1$, we have that $\frac{a+1}{ \gamma}-1 =0$ which in turn implies $b=2$. This then allows a simplification of the result~\eqref{eq:EDO2_sol}, and the solution of Eq.~\eqref{eq:EDO2} can be expressed as
	\begin{equation}\label{eq:EDO2_simple}
		y(x)= \widetilde {\mathcal S}^{\gamma-1,2}_{\gamma,1}[Y](x),\qquad s=\int\frac{dY}{\gamma+g(Y)}.
	\end{equation}
\end{coro}

As a consequence of the latter corollary, and from the definitions of the $\chi$-trigonometric functions and the generalized inverse Sundman transformation, the following theorem is easily obtained.

\begin{teorema}
	 For $\gamma\in\mathbb R$ a real parameter, let us consider the differential equation \eqref{eq:EDO2}, i.e., 
	\begin{equation}\label{eq:EDO_autonomous}
		\frac{yy''}{y'^2}=g(y^\gamma y').
	\end{equation}
	Then, given $\beta\in\mathbb R\setminus[-1/2,0]$, the general solution of the latter equation can be expressed as 
	\begin{equation}\label{eq:sol_inv_beta}
	y(x)=K_3\left({\rm B}_{\chi_g}^{-1}\left(\frac{1}{2\beta\gamma},\,1+\frac1{2\beta};\,\frac{2\beta\gamma}{K_3^{\gamma+1}}\, (x-K_2)  \right)\right)^\frac1{2\beta\gamma},
	\end{equation}
	where  $K_2$ and $K_3$ are arbitrary constants  and $\beta$ is such that $\beta\gamma>0$, and
	\begin{equation}\label{eq:sol_inv_beta(chi)}
		\chi_g(z)=\sqrt{\frac {g\left(z^{1/\beta}\right)}\gamma} \, .
	\end{equation}
In particular, letting $\beta=-1$ for $\gamma<0$, and with $g$ a negative function, the solution of Eq.~\eqref{eq:EDO_autonomous} can be expressed as

\begin{equation}
	y(x)=K_3\;{\rm sa}_{\chi_g}\left(2|\gamma|-1;\; \frac{|\gamma|}{K_3^{\gamma+1}}(x-K_2)\right),
	\end{equation}
with
$\chi_g(z)=\sqrt{\frac {g\left(1/z\right)}\gamma}.$
\end{teorema}

\begin{proof}
We start by recalling that the solution of Eq.~\eqref{eq:EDO_autonomous} is as given in Corollary~\ref{coro:EDO2} through the use of generalized Sundman transformations.  More precisely, letting $a=\beta\gamma-1$ (which in turn implies $b=\beta+1$) in Eq.~\eqref{eq:EDO2_sol},

	\begin{equation}\label{eq:sol_beta}
			y(x)= \widetilde {\mathcal S}^{\beta\gamma-1,\beta+1}_{\beta\gamma,\beta}[Y](x),\qquad s=\frac1\beta\int\frac{dY}{\gamma+g(Y^{1/\beta})}.
		\end{equation}
	 
Note that, assuming $\gamma\neq0$ and $\frac{g\left(Y^{1/\beta}\right)}{\gamma}\geqslant 0$,   the change of variable in the latter equation can be expressed by using the definition of the $\chi$-arctan function in Eq.~\eqref{eq:chi_arctan}, as
\begin{equation}
	s(Y)=\frac1{\beta\gamma}\int\frac{\diff Y}{1+\chi_g^2(Y)}=\frac{1}{\beta\gamma}\arctan_{\chi_g}(Y)+K,
\end{equation}
where
\begin{equation}
	\chi_g(Y)=\sqrt{\frac {g\left(Y^\frac 1\beta\right)}\gamma}. 
\end{equation}
Then,
\begin{equation}
	Y(s)=\tan_{\chi_g}\left(\beta\gamma (s-K)\right)\, .
\end{equation}
We continue by aplying the inverse Sundman transformation  $\widetilde{\mathcal S}^{\beta\gamma-1,1+\beta}_{\beta\gamma,\beta}$ . To do so,  we start by computing the function $\mathcal G$ defined in Eq.~\eqref{eq:invGSTs4}. Recalling that $b-1=\beta$, this is written as
\begin{equation}
	\mathcal G(s)=K_3\exp\left(\int\frac{\diff s}{Y(s)}\right)=K_3\exp\left(\int\frac{\diff s}{\tan_{\chi_g}(\beta\gamma (s-K))}\right)=K_3\exp\left(\frac{1}{\beta\gamma}\int\frac{\diff \overline s}{\tan_{\chi_g}(\overline s)}\right),
\end{equation}
that is,  using \eqref{dlogsinf}, 
\begin{equation}
	\mathcal G(s)=
	K_3\sin^\frac{1}{\beta\gamma}_{\chi_g}\left(\overline s\right)=
	K_3\sin^\frac{1}{\beta\gamma}_{\chi_g}\left(\beta\gamma (s-K)\right)
\end{equation}
(where we note that $K_3$ already represents any constant of integration).

To obtain the change of variable $s=s(x)$ given in the definition of the inverse Sundman transform we have to compute the following integral,
\begin{equation}
	x(s) -K_2  = \int \frac{\diff s}{\mathcal G(s)^{a-\frac {\alpha b}\beta}Y(s)^{\frac b\beta}},
\end{equation}
with $\alpha=\beta\gamma,\,a=\beta\gamma-1,$ and $b=\beta+1$, that is
\begin{align}\label{eq:xs1}
	x(s)  -K_2  &= \int \frac{\mathcal G(s)^{\gamma+1}}{Y(s)^{\frac{\beta+1}\beta}}\, ds=  K_3^{\gamma+1}\int \left(\sin_{\chi_g}(\beta\gamma (s-K))\right)^{\frac{1-\beta\gamma}{\beta\gamma}} \left(\cos_{\chi_g}(\beta\gamma (s-K))\right)^{\frac{\beta+1}\beta}\, \diff s \nonumber \\
	& = \frac{K_3^{\gamma+1}}{\beta\gamma} \int \left(\sin_{\chi_g}(\overline s)\right)^{\frac{1-\beta\gamma}{\beta\gamma}} \left(\cos_{\chi_g}(\overline s)\right)^{\frac{\beta+1}\beta}\, \diff \overline s \, ,
\end{align}
We now use the definition of the incomplete $\chi$-Beta function given in Eq.~\eqref{eq:Betagen} to write the change of variable $x(s)$ as 
\begin{equation}
	x(s)  -K_2  =\frac{K_3^{\gamma+1}}{2\beta\gamma}{\rm B}_{\chi_g}\left(\frac{1}{2\beta\gamma},\,1+\frac1{2\beta};\,\sin_{\chi_g}^2 (\beta\gamma(s-K))\right),
\end{equation}
or equivalently
\begin{equation}
	\beta \gamma  (s(x)-K)  =\arcsin_{\chi_g}\left(\sqrt{ {\rm B}_{\chi_g}^{-1}\left(\frac{1}{2\beta\gamma}\,,1+\frac1{2\beta};\,\frac{2\beta\gamma}{K_3^{\gamma+1}}\,  (x-K_2)  \right)}\right),
\end{equation}
which in turn allows us to write the solution $y(x)$ of Eq.~\eqref{eq:EDO_autonomous} as
\begin{eqnarray}\nonumber
	y(x)&=&\mathcal G(s(x))=K_3\sin_{\chi_g}^{\frac{1}{\beta\gamma}}\left[\arcsin_{\chi_g}\left(\sqrt{{\rm B}_{\chi_g}^{-1}\left(\frac{1}{2\beta\gamma},\,1+\frac1{2\beta};\,\frac{2\beta\gamma}{K_3^{\gamma+1}}\,  (x-K_2)  \right)}\right)\right]
	\\
	&=&K_3\left({\rm B}_{\chi_g}^{-1}\left(\frac{1}{2\beta\gamma},\,1+\frac1{2\beta};\,\frac{2\beta\gamma}{K_3^{\gamma+1}}\,  (x-K_2)  \right)\right)^\frac{1}{2\beta\gamma}.
\end{eqnarray}

The conditions on the parameters $\beta$ and $\gamma$ are imposed in order to ensure that the parameters $a=\frac{1}{2\beta\gamma}$ and $b=1+\frac1{2\beta}$  are both positive.

\end{proof}
\section{Solutions of families of nonlinear equations in terms of the $\chi$-Ateb sine function}

In this last section we apply  our  results to a pair of nonlinear differential equations  of much recent and current interest.  The first,  of great  theoretical interest~\cite{Mahomed2007,Karasu2009, Nakpim2010, Mustafa2013, Orverem2021}, has been studied  over the last few years, including  by using generalized Sundman transformation methods. The second differential equation is a nonlinear extension of the harmonic oscillator  which has also recently been the subject of much intensive research 
~\cite{Cveticanin2020a,Cveticanin2020b,Vujkov2022,Cveticanin2025b}, using both analytical and approximative methods.  Here we  give the exact solution of Eq.~\eqref{eq:nonlinear_osc_f} under certain constraints on the potential function $f$. 

\subsection{A first example: the equation $y'' +\mu_2y^{k_2} y' + \mu_1y^{k_1}=0$}
The equation
\begin{equation}\label{eq:EDOmus}
	y'' +\mu_2y^{k_2} y' + \mu_1y^{k_1}=0,
\end{equation}
with $\mu_1,\mu_2>0$ and $k_1,k_2\in\mathbb{R}$ was studied in ~\cite{Nakpim2010}. It can be written as
\begin{equation}
	\frac{y''y}{y'^2} = -\mu_2 \frac{y^{k_2+1}}{y'}-\mu_1\frac{y^{k_1+1}}{y'^2}=-\mu_2 \frac{y^{k_2+1}}{y'}-\mu_1\left(\frac{y^{\frac{k_1+1}2}}{y'}\right)^2.
\end{equation}
Imposing that $k_2+1=(k_1+1)/2,$ it follows that
\begin{equation}
	\frac{y''y}{y'^2} = -\mu_2 \frac{y^{k+1}}{y'}-\mu_1\left(\frac{y^{k+1}}{y'}\right)^2 :=g\left(\left(y^{-(k+1)}y'\right)^{-1}\right),
\end{equation}
where we have dropped the subindex in $k_2:=k$ and  where 
\begin{equation}
g(z)=-\mu_2z^{-1}-\mu_1z^{-2}=-\frac{\mu_2\,z+\mu_1}{z^2}.
\end{equation}
 
We now take $\beta=-1$ and $\gamma = -(k+1)$ in Eq.~\eqref{eq:sol_inv_beta(chi)}, and so
have that 
$$\chi_g(z)=\sqrt{\frac{g(1/z)}{\gamma}}=\sqrt{\frac{\mu_2 z+\mu_1z^2}{k+1}}:=\eta(z).$$
 We can thus finally 
express the solution as 
\begin{equation}
	y(x)=K_3\left({\rm B}^{-1}_\eta\left(\frac{1}{2k+2},\,\frac1{2};\,2(k+1)K_3^{k}\,  (x-K_2)  \right)\right)^\frac{1}{2k+2},
\end{equation}
that is, 
\begin{equation}
		y(x)=K_3\;{\rm sa}_{\eta}\left(2k+1;\; (k+1)K_3^{k}  (x-K_2)  \right).
	\end{equation}

In the particular case $k_2=k=1,\,k_1=3,\,\mu_1=1,\, \mu_2=3$, Eq.~\eqref{eq:EDOmus} reduces to
\begin{equation}\label{ric2}
	y'' +3y y' +y^3=0,
\end{equation}
which has been widely studied,  see, e.g.,  the references~\cite{Mahomed2007,Karasu2009, Nakpim2010, Mustafa2013, Orverem2021}. Its solution is  then  given by
\begin{equation}\label{gsnho}
		y(x)=K_3\sqrt[4]{{\rm B}^{-1}_\eta\left(\frac{1}4,\,\frac1{2};\, 4 K_3  (x-K_2)  \right)}=K_3\,{\rm sa}_\eta\left(3;\, 2 K_3  (x-K_2)  \right)
\end{equation}
where 
\begin{equation}
	\eta(z)=\sqrt{\frac{3z+z^2}{2}}.
\end{equation}

We note that the above choice of $\eta(z)$ corresponds to the third case 4.3 of our examples of specific choices of $\chi(s)$ discussed in Sections 4.3 and 5.1. 
Making use of Eq.~\eqref{ast}, and defining new arbitrary constants $A$ and $B$ in
terms of $K_2$ and $K_3$ via
$$
K_3^2(A+K_2)^2+1=0,\qquad{} A+B+2K_2=0,
$$
we find that \eqref{gsnho} can be written
\begin{equation}
y(x)=\frac1{x+A}+\frac1{x+B},
\end{equation}
which is well-known to be the general solution of the second order member of the
Riccati hierarchy \eqref{ric2}, i.e., of the case $n=2$ of the sequence of equations
$\displaystyle \left(\frac{d}{dx}+y\right)^n y=0$. Since, setting $K_3=iK_1$, we have
$$
A=\pm\frac{1}{K_1}-K_2,\qquad{} B=\mp\frac{1}{K_1}-K_2,
$$
we see that the particular solutions obtained when (a) $A=B$, (b) $|B|\to\infty$,
and (c) $|A|\to\infty$, $|B|\to\infty$, i.e.
$$
y(x)=\frac{2}{x+A},\qquad{} y(x)=\frac{1}{x+A},\qquad{} y(x)=0,
$$
respectively, can be obtained as the following corresponding special cases of
\eqref{gsnho}: (a) $|K_1|\to\infty$
and $K_2$ finite, (b) $K_2=\pm1/K_1-C$, $C$ a fixed finite constant,
and $K_1\to 0$, (c) $K_1\neq0$ (including $|K_1|\to\infty$) and $|K_2|\to\infty$.

It is interesting to note that, not only do our new functions allow us to express
the solutions of differential equations, but also, our knowledge of the properties 
and special cases of such solutions, and of alternative expressions for them, then
allows us to deduce properties, such as limiting cases, of our new functions.

\subsection{The nonlinear oscillator}
The nonlinear harmonic oscillator with mass variation has recently been studied in~\cite{Cveticanin2025b}. We  begin   with the equation
\begin{equation}\label{eq:nonlinear_osc}
	\ddot x +\omega_\alpha^2 x^\alpha+ \frac{m'(x)}{m(x)} \left(\frac{\dot x^2}{2}+\omega_\alpha^2\frac{x^{\alpha+1}}{\alpha+1}\right)=0,
\end{equation}
or equivalently,
\begin{equation}
	\frac{\ddot x x}{\dot x^2} +\omega_\alpha^2 \frac{x^{\alpha+1}}{\dot x^2}+ \frac{m'(x)}{m(x)} x \left(\frac{1}{2}+\frac{\omega_\alpha^2}{\alpha+1}\frac{x^{\alpha+1}}{\dot x^2}\right)=0,
\end{equation}
 where we use a dot $\dot {}$ to denote derivatives with respect to the 
independent variable $t$. 
In this  subsection  we supposse that the mass variation satisfies $\frac{m' (x)}{m(x)}=\frac kx,$ or equivalently $m(x)=Cx^k.$ Then Eq.~\eqref{eq:nonlinear_osc} reads
\begin{equation}\label{eq:nonlinear_power}
	\frac{\ddot x x}{\dot x^2} +\omega_\alpha^2 \frac{x^{\alpha+1}}{\dot x^2}+ k \left(\frac{1}{2}+\frac{\omega_\alpha^2}{\alpha+1}\frac{x^{\alpha+1}}{\dot x^2}\right)=0,
\end{equation}
that is,
\begin{equation}\label{eq:EDO}
	\frac{\ddot x x}{\dot x^2} =-\omega_\alpha^2\, \frac{k+\alpha+1}{\alpha+1}\frac{x^{\alpha+1}}{\dot x^2}-\frac k2.
\end{equation}
We recall that this equation depends only on the \textit{physical} parameters $\alpha$, $k$  and $\omega_\alpha$. 

Using the notation
\begin{equation}
	g(z):=-Az^{-2}-B=-\frac{A+Bz^2}{z^2},
\end{equation}
with $A=\omega_\alpha^2\, \frac{k+\alpha+1}{\alpha+1}$ and $B=\frac k2,$  Eq.~\eqref{eq:nonlinear_power} is written
\begin{equation}
	\frac{\ddot x x}{\dot x^2} =g\left(x^{-\frac{\alpha+1}2}\dot x\right).
\end{equation}
In addition, letting $\beta=-1$ and $\gamma=-(\alpha+1)/2$
 in Eq.~\eqref{eq:sol_inv_beta(chi)},  we have that 
\begin{equation}
	\frac{g(1/z)}\gamma=\frac{2Az^2+k}{\alpha+1}=\frac{2\,(k+\alpha+1)\,\omega_\alpha^2}{(\alpha+1)^2}\,z^2+\frac k{\alpha+1}.
\end{equation}
 Thus,  denoting 
\begin{equation}
	\zeta(z)=\sqrt{\frac{2\,(k+\alpha+1)\,\omega_\alpha^2}{(\alpha+1)^2}\,z^2+\frac k{\alpha+1}},
\end{equation}
we have that the solution of Eq.~\eqref{eq:nonlinear_power} is given by 
\begin{equation}
x(t)=K_3\left({\rm B}^{-1}_\zeta\left(\frac1{\alpha+1},\,\frac1{2};\,(\alpha+1)K_3^{\frac{\alpha-1}{2}}(t-K_2)\right)\right)^\frac{1}{\alpha+1}.
\end{equation}
That is,
\begin{equation}
	x(t)=K_3\,{\rm sa}_\zeta\left(\alpha;\,\frac{(\alpha+1)K_3^{\frac{\alpha-1}{2}}}{2}\,(t-K_2)\right).
\end{equation}
Additionally, since $\zeta (z)$ has the same structure as the deformation function $\chi_{a,b}$  of the first of our illustrative examples, i.e., case 4.1 (see Sections 4.1 and 5.1),  we have that the incomplete $\zeta$-Beta function is given in terms of the classical incomplete Beta function, as shown in~\eqref{eq:chi_a-b_beta}.

In  the  more general case $\frac{m' (x)}{m(x)}=\frac {\tilde h(x^{\frac{\alpha+1}2}/ \dot x)}x,$ we have 
\begin{equation}
	\frac{\ddot x x}{\dot x^2} +\omega_\alpha^2 \frac{x^{\alpha+1}}{\dot x^2}+ \tilde h(x^{\frac{\alpha+1}2}/ \dot x)\left(\frac{1}{2}+\frac{\omega_\alpha^2}{\alpha+1}\frac{x^{\alpha+1}}{\dot x^2}\right)=0,
\end{equation}
 or  
\begin{equation}
	\frac{\ddot x x}{\dot x^2} 
	=-\frac{\omega_\alpha^2}{\alpha+1}\left( \widetilde h(x^{\frac{\alpha+1}2} /\dot x)+\alpha+1\right)\frac{x^{\alpha+1}}{\dot x^2}-\frac 12  \widetilde h(x^{\frac{\alpha+1}2}/ \dot x).
\end{equation}

We define 
\begin{equation}
g(z):= -\widetilde h(1/z)\left(\frac{\overline A }{z^2}+\frac 12\right)-\frac{\omega_\alpha^2} {z^2},\quad \overline A=\frac{\omega_\alpha^2 }{\alpha+1}.
\end{equation}
Then 
\begin{equation}
	\frac{\ddot x x}{\dot x^2} 
	=g\left(
	x^{-\frac{\alpha+1}2}\dot x
\right).
\end{equation}

We again let $\beta=-1$ and $\gamma=-\frac{\alpha+1}2$  in Eq.~\eqref{eq:sol_inv_beta(chi)}. After  some algebraic manipulations it follows that
%
%
%
%
%

\begin{equation}
	\sqrt{\frac{g(1/z)}{\gamma}}=\frac {\sqrt{(2\omega_\alpha^2\,z^2+\alpha+1)\,\widetilde h(z)+  2 \omega_\alpha^2(\alpha+1)z^2}}{\alpha+1}:=\Phi(z).
\end{equation}
Finally,  we obtain 
\begin{equation}
		x(t)=K_3 \,{\rm sa}_\Phi\left(
		\alpha;\frac{(\alpha+1)K_3^{\frac{\alpha-1}{2}}}{2}\,(t-K_2)\right).
	\end{equation}

\begin{remark}[The "general" case]
The methodology described in Subsection  5.2  
is applicable to  other  differential equations  having  the form of Eq.~\eqref{eq:nonlinear_osc_f} studied in~\cite{Cveticanin2020a}. 
More precisely, one can rewrite this equation as
$$
\frac{\ddot x x}{(\dot x)^2}=\frac{x}{(\dot x)^2}f(x,\dot x)-\omega_\alpha ^2\frac{x^{\alpha+1}}{(\dot x)^2}.
$$
Then, in the case that the function $f$ has the form 
$$
f(x,\dot x)=\frac{(\dot x)^2}x F(x^{-\frac{\alpha+1}2}\dot x),
$$
for  some  arbitrary function $F$,  we have that  the solution of Eq.~\eqref{eq:nonlinear_osc_f} is given in terms of the $\chi_g$-Ateb sine function with 
$$
g(z)=F(z)-\frac{\omega_\alpha^2}{z^2}.
$$
 Moreover,  if the function $f$ has the form
$$
f(x,\dot x)=\frac{(\dot x)^2}y F(x^{\gamma}\dot x)+\omega_\alpha^2 x^\alpha,
$$
for an arbitrary real number $\gamma$ and  some  arbitrary function $F$, then the solution of equation~\eqref{eq:nonlinear_osc_f} can be expressed in terms of the $\chi_F$-Ateb sine function. 
\end{remark}

\section{Conclusions}
In this  paper  we propose a generalization of the main trigonometric functions, beyond  the parametric extension of Linqvist, through the use of a positive function as a deformation generator, in  a way that is consistent  with previous extensions of the logarithmic and exponential functions. We study the main basic mathematical properties of  our new  $\chi$-trigonometric functions, as well as their relation with the previous $\chi$-logarithmic functions, which in turn allow us to establish the corresponding Pytheagorean-like relations.  These last, in fact,  include the fundamental trigonometric relation and some known extensions, but also other fundamental relations between special functions such as the logarithmic and exponential integrals. 

 We also define  a new extension of the incomplete Beta function,  giving, indeed,  two alternative definitions:  firstly,  by using the $\chi$-trigonometric functions, and  secondly,  by using the $\chi$-exponential function. We show their  usefulness in expressing  the general solution of a family of nonlinear autonomous differential equations, including those related to the nonlinear harmonic oscillator. The whole family of solutions,  which, it is worth remarking, includes  an arbitrary function as a degree of freedom, can be expressed through an extension of the Ateb sine function defined  using our  new $\chi$-Beta function.

The new $\chi$-deformed functions introduced here would seem to be worthy of further
study, 
since they allow us to obtain closed-form expressions for the general solution of a
relevant family of nonlinear differential equations  having   both theoretical and applied interest.  Further applications will be considered in later papers.

\section*{Acknowledgements}

The authors thank the Universidad Rey Juan Carlos for funding through the project
2025/SOLCON-160677. PRG and AP also thank the Universidad Rey Juan Carlos for 
funding as members of the Grupo de Investigaci\'on de Alto Rendimiento DELFO.

\section*{Conflict of Interest Statement}
This work does not have any conflicts of interest.

	\bibliographystyle{unsrt}
\bibliography{refsbis}

\end{document}